\documentclass[a4paper,english,sumlimits,reqno]{amsart}

\usepackage[T1]{fontenc}
\usepackage[utf8]{inputenc}
\usepackage{lmodern}

\usepackage{babel}

\usepackage{xifthen}
\usepackage{ifthenx}

\usepackage{xargs}

\usepackage{graphicx}
\usepackage[table,svgnames,x11names]{xcolor}

\usepackage[colorinlistoftodos,prependcaption,textsize=tiny]{todonotes}
\newcommandx{\todoin}[2][1=]{\todo[inline, caption={todo}, #1]{%
\begin{minipage}{\textwidth-20pt}#2\end{minipage}}}

\usepackage{amssymb}
\usepackage{amsmath}
\usepackage{amsthm}
\usepackage{mathtools}
\usepackage{exscale}
\usepackage{relsize}
\usepackage{bbm}
\usepackage{bm}
\usepackage{amsbsy}

\usepackage[mathcal]{eucal}
\usepackage{mathrsfs}

\usepackage{stmaryrd}

\usepackage[all]{xy}
\newcommand{\vcxymatrix}[1]{\vcenter{\xymatrix{#1}}}

\usepackage{fp}

\usepackage[section]{placeins}
\usepackage{float}

\usepackage{enumerate}

\newcounter{proof}
\newenvironment{myproof}%
{\stepcounter{proof}\begin{proof}}%
{\end{proof}}%
\newcounter{proofstep}[proof]
\newenvironment{proofstep}[1][]%
{\refstepcounter{proofstep}\bigskip\par\noindent%
  \ifthenelse{\isempty{#1}}
    {\textsc{Step \theproofstep. }}
    {\textsc{#1.}}
  \noindent}%
{\par}%
\newcounter{proofcase}[proof]
{\refstepcounter{proofcase}\bigskip\par\noindent%
  \ifthenelse{\isempty{#1}}
    {\textsc{Case \theproofcase. }}
    {\textsc{#1.}}
  \noindent}%
{\par}%

\usepackage{hyperref}

\theoremstyle{plain}
\newtheorem{thm}{Theorem}[section]
\newtheorem*{thm*}{Theorem}

\newtheorem{cor}[thm]{Corollary}
\newtheorem{lem}[thm]{Lemma}

\theoremstyle{definition}
\newtheorem{dfn}[thm]{Definition}
\theoremstyle{remark}

\newtheorem{rem}[thm]{Remark}

\numberwithin{equation}{section}

\newcommandx{\textref}[2][1=]{\hyperref[#2]{#1\ref*{#2}}}
\newcommandx{\textrefp}[2][1=]{(\hyperref[#2]{#1\ref*{#2}})}

\DeclareMathOperator*{\Union}{\bigcup}

\newcommand{\cc}[1]{\ensuremath{\llbracket #1 \rrbracket}}

\newcommand{\dif}{\ensuremath{\, \mathrm d}}

\DeclareMathOperator{\Id}{Id}

\DeclareMathOperator{\spn}{span}

\DeclareMathOperator{\cond}{\mathbb{E}}

\begin{document}

\title[Direct sums of finite dimensional $SL^\infty_n$ spaces]{Direct sums of finite dimensional
  \bm{$SL^\infty_n$} spaces}

\author[R.~Lechner]{Richard Lechner}

\address{Richard Lechner, Institute of Analysis, Johannes Kepler University Linz, Altenberger
  Strasse 69, A-4040 Linz, Austria}

\email{Richard.Lechner@jku.at}

\date{\today}



\thanks{Supported by the Austrian Science Foundation (FWF) Pr.Nr.  P28352}

\begin{abstract}
  $SL^\infty$ denotes the space of functions whose square function is in $L^\infty$, and the
  subspaces $SL^\infty_n$, $n\in\mathbb{N}$, are the finite dimensional building blocks of
  $SL^\infty$.

  We show that the identity operator $\Id_{SL^\infty_n}$ on $SL^\infty_n$ well factors through
  operators $T : SL^\infty_N\to SL^\infty_N$ having large diagonal with respect to the standard Haar
  system.  Moreover, we prove that $\Id_{SL^\infty_n}$ well factors either through any given
  operator $T : SL^\infty_N\to SL^\infty_N$, or through $\Id_{SL^\infty_N}-T$.  Let $X^{(r)}$ denote
  the direct sum $\bigl(\sum_{n\in\mathbb{N}_0} SL^\infty_n\bigr)_r$, where $1\leq r \leq \infty$.
  Using Bourgain's localization method, we obtain from the finite dimensional factorization result
  that for each $1\leq r\leq \infty$, the identity operator $\Id_{X^{(r)}}$ on $X^{(r)}$ factors
  either through any given operator $T : X^{(r)}\to X^{(r)}$, or through $\Id_{X^{(r)}} - T$.
  Consequently, the spaces $\bigl(\sum_{n\in\mathbb{N}_0} SL^\infty_n\bigr)_r$,
  $1\leq r\leq \infty$, are all primary.
\end{abstract}

\maketitle

\makeatletter
\providecommand\@dotsep{5}
\def\listtodoname{List of Todos}
\def\listoftodos{\@starttoc{tdo}\listtodoname}
\makeatother

\section{Introduction}\label{sec:intro}

\noindent
Let $\mathcal{D}$ denote the collection of \emph{dyadic intervals} contained in the unit
interval~$[0,1)$; it is given by
\begin{equation*}
  \mathcal{D} = \{[(k-1)2^{-n},k2^{-n}) : n\in \mathbb{N}_0, 1\leq k\leq 2^n\}.
\end{equation*}
Let $|\cdot|$ denote the Lebesgue measure.  For any $N\in\mathbb{N}_0$ we put
\begin{equation}\label{eq:dyadic-intervals}
  \mathcal{D}_N = \{I\in\mathcal{D} : |I|=2^{-N}\}
  \qquad\text{and}\qquad
  \mathcal{D}^N = \bigcup_{n=0}^N\mathcal{D}_n.
\end{equation}

We denote the $L^\infty$-normalized \emph{Haar function} supported on $I\in\mathcal{D}$ by $h_I$;
i.e. if $I_0,I_1\in\mathcal{D}$ are such that $\inf I_0 < \inf I_1$ and $I_0\cup I_1=I$, then
\begin{equation*}
  h_I=\chi_{I_0}-\chi_{I_1},
\end{equation*}
where $\chi_A$ denotes the characteristic function of $A\subset [0,1)$.  The \emph{Rademacher
  functions} $r_n$, $n\in\mathbb{N}_0$ are given by
\begin{equation*}
  r_n = \sum_{I\in\mathcal{D}_n} h_I,
  \qquad n\in\mathbb{N}_0.
\end{equation*}

The \emph{non-separable} Banach space $SL^\infty$ is given by
\begin{equation}\label{eq:sl-infty:space}
  SL^\infty = \Big\{f = \sum_{I\in\mathcal{D}} a_I h_I\in L^2 : \|f\|_{SL^\infty} < \infty\Big\},
\end{equation}
equipped with the norm
\begin{equation}\label{eq:sl-infty:norm}
  \Big\| \sum_{I\in\mathcal{D}} a_I h_I \Big\|_{SL^\infty}
  = \Big\| \big(\sum_{I\in\mathcal{D}} a_I^2 h_I^2\big)^{1/2} \Big\|_{L^\infty}.
\end{equation}
(To see that $SL^\infty$ is non-separable, consider any infinite collection of pairwise disjoint
dyadic intervals $\{ I_j : j\in\mathbb{N}\}\subset \mathcal{D}$, and embed $\ell^\infty$ into
$SL^\infty$ by $e_j\mapsto h_{I_j}$, where $e_j$ denotes the $j^{\text{th}}$ standard unit vector in
$\ell^\infty$.)  We want to emphasize that throughout this paper, any series in $SL^\infty$ merely
represents the vector of coefficients, and it does \emph{not} indicate any kind of convergence.  For
variants of the space $SL^\infty$, we refer the reader to~\cite{jones:mueller:2004}.  The
\emph{Hardy space} $H^1$ is the completion of
\begin{equation}\label{eq:Hp-space}
  \spn\{ h_I : I \in \mathcal{D} \}
\end{equation}
under the norm
\begin{equation}\label{eq:Hp-norm}
  \big\| \sum_{I\in \mathcal{D}} a_I h_I \big\|_{H^1}
  = \int_0^1 \Big(
  \sum_{I\in \mathcal{D}} a_I^2 h_I^2(x)
  \Big)^{1/2}
  \dif x
  .
\end{equation}

We define the \emph{duality pairing}
$\langle \cdot, \cdot \rangle : SL^\infty\times H^1\to \mathbb{R}$ by
\begin{equation}\label{eq:scalar-product}
  \langle f, g\rangle = \int_0^1 f(x)g(x)\dif x,
  \qquad f\in SL^\infty,\ g\in H^1.
\end{equation}
and note the well-known and obvious inequality (see e.g.~\cite{garsia:1973}):
\begin{equation}\label{eq:bracket-estimate}
  |\langle f, g\rangle| \leq \|f\|_{SL^\infty} \|g\|_{H^1},
  \qquad f\in SL^\infty,\ g\in H^1.
\end{equation}

We call a bounded linear map between Banach spaces an \emph{operator}.
By~\eqref{eq:bracket-estimate}, the operator $J : SL^\infty\to (H^1)^*$ defined by
$g\mapsto \bigl(f\mapsto \langle f, g \rangle \bigr)$ has norm $1$.  Now, let
$I : H^1\to (H^1)^{**}$ denote the canonical embedding.  Hence, for any given operator
$T : SL^\infty\to SL^\infty$, the operator $S : H^1\to (SL^\infty)^*$, defined by $S:=T^*J^*I$ is
bounded by $\|T\|$ and satisfies
\begin{equation}\label{eq:adjoint}
  \langle Sg, f\rangle = \langle g, Tf\rangle,
  \qquad f\in SL^\infty,\ g\in H^1,
\end{equation}
where the duality pairing on the left hand side is the canonical duality pairing between
$(SL^\infty)^*$ and $SL^\infty$, and the duality pairing on the right hand side is given
by~\eqref{eq:scalar-product}.  For the sake of brevity, we shall write $T^*$ instead of $T^*J^*I$.

Given $n\in\mathbb{N}_0$, we define the following finite dimensional spaces:
\begin{equation}\label{eq:finite-dimensional-spaces}
  SL^\infty_n = \spn\{h_I : I\in\mathcal{D}^n\}\subset SL^\infty
  \quad\text{and}\quad
  H^1_n = \spn\{h_I : I\in\mathcal{D}^n\}\subset H^1.
\end{equation}
Let $n\in\mathbb{N}_0$, $\delta > 0$, and let $T : SL^\infty_n\to SL^\infty_n$ denote an operator.
We say that $T$ has a \emph{$\delta$-large diagonal} with respect to the Haar system
$(h_I : I\in\mathcal{D}^n)$ if
\begin{equation}\label{eq:large-diag}
  |\langle T h_I, h_I \rangle|
  \geq \delta |I|,
  \qquad I\in \mathcal{D}^n.
\end{equation}
If unambiguous, we simply say $T$ has large diagonal without explicitly specifying $\delta$ and the
system of functions.

\section{Main Results}\label{sec:results}

\noindent
Theorem~\ref{thm:local-factor} asserts that the identity operator on $SL^\infty_n$ factors through any
operator $T : SL^\infty_N\to SL^\infty_N$ having large diagonal, where $N$ depends (among other
parameters) on $n$ and $\|T\|$.  It is the first step towards a factorization result for direct sums
of $SL^\infty_n$ spaces.  Theorem~\ref{thm:local-factor} is a finite dimensional quantitative version of
the infinite dimensional factorization result~\cite[Theorem~2.1]{lechner:2016:factor-SL}.
\begin{thm}\label{thm:local-factor}
  Let $n\in\mathbb{N}_0$, $\Gamma,\eta > 0$ and $\delta > 0$.  Then there exists an integer
  $N = N(n,\Gamma,\eta,\delta)$, such that for any operator $T : SL^\infty_N\to SL^\infty_N$ with
  $\|T\|\leq \Gamma$ and
  \begin{equation*}
    |\langle T h_K, h_K \rangle| \geq \delta |K|,
    \qquad K\in \mathcal{D}^N,
  \end{equation*}
  there exist operators $R : SL^\infty_n\to SL^\infty_N$ and $S : SL^\infty_N\to SL^\infty_n$ such
  that the diagram
  \begin{equation}\label{eq:factor}
    \vcxymatrix{SL^\infty_n \ar[r]^\Id \ar[d]_R & SL^\infty_n\\
      SL^\infty_N \ar[r]_T & SL^\infty_N \ar[u]_S}
  \end{equation}
  is commutative.  Moreover, the operators $R$ and $S$ can be chosen in such way that they satisfy
  $\|R\|\|S\| \leq (1+\eta)/\delta$.
\end{thm}

Our next result is the local factorization Theorem~\ref{thm:local-primary}.  The key difference
between Theorem~\ref{thm:local-primary} and Theorem~\ref{thm:local-factor} is that in
Theorem~\ref{thm:local-primary} we do not require the operator $T$ to have large diagonal.  To
compensate, we will use additional combinatorics to select a large subset of intervals
$\mathcal{L}$, on which either $T$ or $\Id-T$ has large diagonal (see
Section~\ref{subsec:local-primary-prep} and~\ref{subsec:local-primary}).  It is the choice of that
$\mathcal{L}$ which determines whether the identity factors through $T$ or $\Id-T$ (see
Section~\ref{subsec:local-primary}).
\begin{thm}\label{thm:local-primary}
  Let $n\in\mathbb{N}_0$ and $\Gamma,\eta > 0$.  Then there exists an integer
  $N = N(n,\Gamma,\eta)$, such that for any operator $T : SL^\infty_N\to SL^\infty_N$ with
  $\|T\|\leq \Gamma$, we can find operators $R : SL^\infty_n\to SL^\infty_N$ and
  $S : SL^\infty_N\to SL^\infty_n$, such that for either $H=T$ or $H=\Id-T$ the diagram
  \begin{equation}\label{eq:thm:local-primary}
    \vcxymatrix{SL^\infty_n \ar[r]^\Id \ar[d]_R & SL^\infty_n\\
      SL^\infty_N \ar[r]_H & SL^\infty_N \ar[u]_S}
  \end{equation}
  is commutative.  Moreover, it is possible to choose the operators $R$ and $S$ in such way that
  they satisfy $\|R\|\|S\| \leq 2+\eta$.
\end{thm}

Recall that a Banach space $X$ is \emph{primary} if for every bounded projection $Q : X\to X$,
either $Q(X)$ or $(\Id - Q)(X)$ is isomorphic to $X$ (see e.g.~\cite{lindenstrauss-tzafriri:1977}).
In Theorem~\ref{thm:primary} we tie the local results of Theorem~\ref{thm:local-primary} for the spaces
$SL^\infty_n$, $n\in\mathbb{N}_0$ together, to obtain factorization results in
$X^{(r)}=\bigl(\sum_{n\in\mathbb{N}_0} SL^\infty_n \bigr)_r$, $1 \leq r \leq \infty$.  Specifically,
we obtain that all the spaces $X^{(r)}$, $1 \leq r \leq \infty$, are primary.  Moreover, in
Section~\ref{subsec:isos}, we will show that $SL^\infty$ is isomorphic to $X^{(\infty)}$; consequently,
$SL^\infty$ is primary, as well.
\begin{thm}\label{thm:primary}
  Let $1 \leq r \leq \infty$, put $X^{(r)}=\bigl(\sum_{n\in\mathbb{N}_0} SL^\infty_n \bigr)_r$ and
  let $T : X^{(r)}\to X^{(r)}$ denote an operator.  Then for any $\eta > 0$, there exist operators
  $R, S : X^{(r)}\to X^{(r)}$ such that for either $H = T$ or $H = \Id_{X^{(r)}} - T$ the diagram
  \begin{equation}\label{eq:thm:primary}
    \vcxymatrix{
      X^{(r)} \ar[rr]^{\Id_{X^{(r)}}} \ar[d]_R && X^{(r)}\\
      X^{(r)} \ar[rr]_H && X^{(r)} \ar[u]_S
    }
  \end{equation}
  is commutative.  The operators $R$ and $S$ can be chosen such that $\|R\| \|S\| \leq 1 + \eta$.
  Consequently, the spaces $\bigl(\sum_{n\in\mathbb{N}_0} SL^\infty_n \bigr)_r$,
  $1\leq r\leq \infty$, as well as $SL^\infty$ are all primary.
\end{thm}

We remark that the primarity of $SL^\infty$ has previously been established
in~\cite{lechner:2016:factor-SL}, by working directly in the non-separable space $SL^\infty$ using
infinite dimensional methods instead of Bourgain's localization method.  In contrast, here we will
use Bourgain's localization method to show the primarity of $SL^\infty$.

\section{Embeddings, projections and quantitative diagonalization of
  operators}\label{sec:projections}

\noindent
The main point of this section is to provide the technical main result (see Theorem~\ref{thm:quasi-diag})
of this paper.  Theorem~\ref{thm:quasi-diag} permits us to quantitatively almost-diagonalize a given
operator $T$ by a block basis of the Haar system.  Moreover, it is possible to select the block
basis in such way, that if the operator $T$ has large diagonal with respect to the Haar system, then
$T$ has large diagonal with respect to the block basis.

Before we come to the proof of Theorem~\ref{thm:quasi-diag}, we will discuss several results on embeddings
and projections in $SL^\infty$ established in~\cite{lechner:2016:factor-SL}, which will play a vital
role in the proof of Theorem~\ref{thm:quasi-diag}.  Additionally, we replace the techniques involving
qualitative limits of Rademacher functions in~\cite{lechner:2016:factor-SL} with quantitative
combinatorics of dyadic intervals (see Lemma~\ref{lem:comb-1}).

\subsection{Embeddings and projections}\label{subsec:compat}\hfill\\

\noindent
Here, we briefly discuss the conditions~\textrefp[J]{enu:j1}--\textrefp[J]{enu:j4} (which go back to
Jones~\cite{jones:1985}) and their consequences.  First, we will show that the
conditions~\textrefp[J]{enu:j1}--\textrefp[J]{enu:j4} are stable under reiteration (see
Theorem~\ref{thm:projection-iteration}).  Then we will prove that whenever a block basis
$(b_I : I\in\mathcal{D}^n)$ of the Haar system $(h_I : I\in\mathcal{D}^N)$ is selected according
to~\textrefp[J]{enu:j1}--\textrefp[J]{enu:j4}, $(b_I : I\in\mathcal{D}^n)$ spans a complemented copy
of $SL^\infty_n$ (see Theorem~\ref{thm:projection}; the constants for the norms of the isomorphism and the
projection do not depend on $n$).

Let $\mathcal{I}\subset \mathcal{D}$ be a collection of dyadic intervals, and let $\mathcal{N}$ be a
collection of sets.  Let $\mathcal{I}$ index collections $\mathcal{B}_I\subset \mathcal{N}$,
$I\in\mathcal{I}$, and put
\begin{equation}\label{eq:abbreviations}
  \mathcal{B}
  = \bigcup_{I\in \mathcal{I}} \mathcal{B}_I
  \qquad\text{and}\qquad
  \qquad B_I = \bigcup \mathcal{B}_I,
  \quad\text{for all $I\in \mathcal{I}$}.
\end{equation}
We say that the (possibly finite) sequence $(\mathcal{B}_I : I \in \mathcal{I})$ \emph{satisfies
  Jones' compatibility conditions} (see~\cite{jones:1985}) \emph{with constant $\kappa_J\geq 1$}, if
the following conditions~\textrefp[J]{enu:j1}--\textrefp[J]{enu:j4} are satisfied:
\begin{enumerate}[\quad(J1)]
\item\label{enu:j1} The collection $\mathcal{N}$ consists of measurable sets with finite and
  positive measure and is nested, i.e.  whenever $N_0,N_1\in\mathcal{N}$ with
  $N_0\cap N_1\neq \emptyset$, then $N_0\subset N_1$ or $N_1\subset N_0$.  Moreover, for each
  $I\in\mathcal{I}$, the collection $\mathcal{B}_I\subset\mathcal{N}$ is finite.

\item\label{enu:j2} For each $I\in\mathcal{I}$, the collection $\mathcal{B}_I$ is non-empty and
  consists of pairwise disjoint sets.  Furthermore,
  $\mathcal{B}_{I_0} \cap \mathcal{B}_{I_1} = \emptyset$, whenever $I_0,I_1\in\mathcal{I}$ are
  distinct.

\item\label{enu:j3} For all $I_0, I_1\in \mathcal{I}$ holds that
  \begin{equation*}
    B_{I_0}\cap B_{I_1} = \emptyset\ \text{if}\ I_0 \cap I_1 = \emptyset,
    \qquad\text{and}\qquad
    B_{I_0} \subset B_{I_1}\ \text{if}\ I_0 \subset I_1.
  \end{equation*}

\item\label{enu:j4} For all $I_0,I\in \mathcal{I}$ with $I_0\subset I$ and $N\in \mathcal{B}_I$, we
  have
  \begin{equation*}
    \frac{|N\cap B_{I_0}|}{|N|} \geq  \kappa_J^{-1} \frac{|B_{I_0}|}{|B_I|}.
  \end{equation*}
\end{enumerate}

In the following Lemma~\ref{lem:projection-simple}, we record three facts about collections
satisfying~\textrefp[J]{enu:j1}--\textrefp[J]{enu:j4}.  It is a straightforward finite dimensional
adaptation of~\cite[Lemma~3.1]{lechner:2016:factor-SL}.
\begin{lem}\label{lem:projection-simple}
  Let $n\in\mathbb{N}_0$ and let $(\mathcal{B}_I : I\in \mathcal{D}^n)$
  satisfy~\textrefp[J]{enu:j1}--\textrefp[J]{enu:j4}.  Then the following statements are true:
  \begin{enumerate}[(i)]
  \item\label{enu:projection-simple:1} $(B_I : I\in \mathcal{D}^n)$ is a finite sequence of nested
    measurable sets of finite positive measure.
    
  \item\label{enu:projection-simple:2} Let $I, I_0\in \mathcal{D}^n$, then
    \begin{equation*}
      B_{I_0} \subset B_I
      \quad\text{if and only if}\quad
      I_0 \subset I.
    \end{equation*}
    
  \item\label{enu:projection-simple:3} Let $I_0, I\in \mathcal{D}^n$, with $I_0\subset I$.  Then for
    all $N_0\in \mathcal{B}_{I_0}$ there exists a set $N\in \mathcal{B}_I$ such that $N_0\subset N$.
  \end{enumerate}
\end{lem}

\begin{proof}
  Replacing $\mathcal{D}^n$ with $\mathcal{D}$ in the proof
  of~\cite[Lemma~3.1]{lechner:2016:factor-SL} and repeating it, yields the above result.
\end{proof}

\begin{rem}\label{rem:projection-simple}
  By Lemma~\ref{lem:projection-simple}~\eqref{enu:projection-simple:2}, we can uniquely identify $I$ with
  $B_I$.
\end{rem}

The following Theorem~\ref{thm:projection-iteration} (which is a finite dimensional version
of~\cite[Theorem~3.2]{lechner:2016:factor-SL}) asserts that Jones' compatibility
conditions~\textrefp[J]{enu:j1}--\textrefp[J]{enu:j4} are stable under iteration.
\begin{thm}\label{thm:projection-iteration}
  Let $n,N\in\mathbb{N}_0$, and let $(\mathcal{A}_I : I \in \mathcal{D}^N)$ be a finite sequence of
  collections of sets that satisfies~\textrefp[J]{enu:j1}--\textrefp[J]{enu:j4} with constant
  $\kappa_J\geq 1$.  Put $\mathcal{M} = \Union_{I\in \mathcal{D}^N}\mathcal{A}_I$ and
  $A_I = \bigcup \mathcal{A}_I$, $I\in\mathcal{D}^N$.  Let $\mathcal{N}$ denote the collection of
  nested sets given by
  \begin{equation*}
    \mathcal{N} = \{A_I : I\in \mathcal{D}^N\}.
  \end{equation*}
  For each $J\in \mathcal{D}^n$ let $\mathcal{B}_J \subset \mathcal{N}$ be such that
  $(\mathcal{B}_J : J \in \mathcal{D}^n)$ satisfies~\textrefp[J]{enu:j1}--\textrefp[J]{enu:j4} with
  constant $\kappa_J\geq 1$, where we put $B_J = \bigcup\mathcal{B}_J$.  Finally, for all
  $J\in \mathcal{D}^n$, we define
  \begin{equation*}
    \mathcal{C}_J = \bigcup_{A_I\in \mathcal{B}_J} \mathcal{A}_I
    \quad\text{and}\quad
    C_J = \bigcup \mathcal{C}_J,
  \end{equation*}
  and we note that $C_J = B_J$.  Then $(\mathcal{C}_J : J\in \mathcal{D}^n)$ is a finite sequence of
  collections of sets in $\mathcal{M}$ satisfying~\textrefp[J]{enu:j1}--\textrefp[J]{enu:j4} with
  constant $\kappa_J^2$.
\end{thm}

\begin{proof}
  The proof of Theorem~\ref{thm:projection-iteration} follows immediately by replacing $\mathcal{D}^n$ with
  $\mathcal{D}$ in the proof of~\cite[Theorem~3.2]{lechner:2016:factor-SL}, and repeating the
  argument.
\end{proof}

\noindent
Here we establish that if $(\mathcal{B}_I : I\in \mathcal{D}^n)$ with
$\mathcal{B}_I\subset\mathcal{D}^N$, $I\in\mathcal{D}^n$, satisfies Jones' compatibility
conditions~\textrefp[J]{enu:j1}--\textrefp[J]{enu:j4}, then the block basis
$(b_I : I\in \mathcal{D}^n)$ given by
\begin{equation*}
  b_I = \sum_{K\in \mathcal{B}_I} h_K,
  \qquad I\in \mathcal{D}
\end{equation*}
spans a complemented copy of $SL^\infty_n$ (the constants for the norms of the isomorphism and the
projection do not depend on $n$).  Theorem~\ref{thm:projection} is a straightforward finite dimensional
adaptation of~\cite[Theorem~3.3]{lechner:2016:factor-SL}
\begin{thm}\label{thm:projection}
  Let $n,N\in\mathbb{N}_0$ and $\mathcal{B}_I\subset \mathcal{D}^N$, $I\in \mathcal{D}^n$.  Assume
  that the finite sequence of collections of dyadic intervals $(\mathcal{B}_I : I\in \mathcal{D}^n)$
  satisfies Jones' compatibility conditions~\textrefp[J]{enu:j1}--\textrefp[J]{enu:j4} with constant
  $\kappa_J\geq 1$.  Let $(b_I : I\in \mathcal{D}^n)$ denote the block basis of the Haar system
  $(h_I : I\in \mathcal{D}^N)$ given by
  \begin{equation}\label{eq:thm:projection:block-basis}
    b_I = \sum_{K\in \mathcal{B}_I} h_K,
    \qquad I\in \mathcal{D}^n.
  \end{equation}
  Then the operators $B : SL^\infty_n\to SL^\infty_N$ and $Q : SL^\infty_N\to SL^\infty_n$ given by
  \begin{equation}\label{eq:thm:projection:operators}
    Bf = \sum_{I\in \mathcal{D}^n} \frac{\langle f, h_I\rangle}{\|h_I\|_2^2} b_I
    \qquad\text{and}\qquad
    Qg = \sum_{I\in \mathcal{D}^n} \frac{\langle g, b_I\rangle}{\|b_I\|_2^2} h_I,
  \end{equation}
  satisfy the estimates
  \begin{equation}\label{eq:thm:projection:estimates}
    \|B f \|_{SL^\infty_N} \leq \|f\|_{SL^\infty_n}
    \qquad\text{and}\qquad
    \|Q g \|_{SL^\infty_n} \leq \kappa_J^{1/2} \|g\|_{SL^\infty_N},
  \end{equation}
  for all $f\in SL^\infty_n$, $g\in SL^\infty_N$.  Moreover, the diagram
  \begin{equation}\label{eq:thm:projection:diagram}
    \vcxymatrix{SL^\infty_n \ar[rr]^{\Id_{SL^\infty_n}} \ar[rd]_{B} & & SL^\infty_n\\
      &  SL^\infty_N  \ar[ru]_{Q} &
    }
  \end{equation}
  is commutative.  Consequently, the range of $B$ is complemented by the projection $BQ$ with
  $\|BQ\|\leq\kappa_J^{1/2}$, and $B$ is an isomorphism onto its range with
  $\|B\| \|B^{-1}\|\leq \kappa_J^{1/2}$.
\end{thm}

\begin{proof}
  Clearly, the finite dimensional operators $B$ and $Q$ in the above theorem are truncated versions
  of the corresponding infinite dimensional operators in~\cite[Theorem~3.3]{lechner:2016:factor-SL}.
  Hence, the result follows by $1$-unconditionality of the Haar system in $SL^\infty$.
\end{proof}

\begin{rem}\label{rem:projection}
  Let $n,N\in\mathbb{N}_0$ and $\mathcal{B}_I\subset\mathcal{D}^N$, $I\in\mathcal{D}^n$ be such that
  $(\mathcal{B}_I : I \in \mathcal{D}^n)$ satisfies Jones' compatibility
  conditions~\textrefp[J]{enu:j1}--\textrefp[J]{enu:j4} with constant $\kappa_J\geq 1$.  Recall that
  in~\eqref{eq:abbreviations} we defined $\mathcal{B} = \bigcup_{I\in \mathcal{D}^n} \mathcal{B}_I$.
  Now, given a finite sequence of signs
  $\varepsilon = (\varepsilon_K\in \{\pm 1\} : K\in \mathcal{B})$ we put
  \begin{equation}\label{eq:block-basis}
    b_I^{(\varepsilon)} = \sum_{K\in \mathcal{B}_I} \varepsilon_K h_K,
    \qquad I\in\mathcal{D}^n.
  \end{equation}
  We call $(b_I^{(\varepsilon)} : I\in\mathcal{D}^n)$ the \emph{block basis generated by}
  $(\mathcal{B}_I : I\in \mathcal{D}^n)$ and
  $\varepsilon = (\varepsilon_K\in \{\pm 1\} : K\in \mathcal{B})$.
 
  The block basis $(b_I^{(\varepsilon)} : I\in\mathcal{D}^n)$ gives rise to the operators
  $B^{(\varepsilon)} : SL^\infty_n\to SL^\infty_N$ and
  $Q^{(\varepsilon)} : SL^\infty_N\to SL^\infty_n$:
  \begin{equation}\label{rem:projection:operators:signs}
    B^{(\varepsilon)} f
    = \sum_{I\in \mathcal{D}^n} \frac{\langle f, h_I\rangle}{\|h_I\|_2^2} b_I^{(\varepsilon)}
    \qquad\text{and}\qquad
    Q^{(\varepsilon)} g
    = \sum_{I\in \mathcal{D}^n}
    \frac{\langle g, b_I^{(\varepsilon)}\rangle}{\|b_I^{(\varepsilon)}\|_2^2} h_I,
  \end{equation}
  $f\in SL^\infty_n$, $g\in SL^\infty_N$.  Before we proceed, recall the definitions of the
  operators $B,Q$ (see Theorem~\ref{thm:projection}).  By the $1$-unconditionality of the Haar system in
  $SL^\infty$ and
  \begin{equation*}
    Q^{(\varepsilon)} g
    = Q g^{(\varepsilon)},
    \qquad \text{for all}\ g\in SL^\infty_N\ \text{and}\ 
    g^{(\varepsilon)}
    = \sum_{K\in \mathcal{B}} \varepsilon_K \frac{\langle g, h_K\rangle}{\|h_K\|_2^2} h_K,
  \end{equation*}
  we obtain
  \begin{equation}\label{rem:projection:estimates:signs}
    \|B^{(\varepsilon)} \| \leq \|B\|
    \qquad\text{and}\qquad
    \|Q^{(\varepsilon)} \| \leq \|Q \|.
  \end{equation}
  Moreover, we have the identity
  \begin{equation}\label{rem:projection:identity:signs}
    Q^{(\varepsilon)} B^{(\varepsilon)} = \Id_{SL^\infty_n}.
  \end{equation}
  Consequently, the range of $B^{(\varepsilon)}$ is complemented by the projection
  $B^{(\varepsilon)}Q^{(\varepsilon)}$ with
  $\|B^{(\varepsilon)}Q^{(\varepsilon)}\|\leq\kappa_J^{1/2}$, and $B^{(\varepsilon)}$ is an
  isomorphism onto its range with
  $\|B^{(\varepsilon)}\| \|(B^{(\varepsilon)})^{-1}\|\leq \kappa_J^{1/2}$.
\end{rem}

\subsection{A combinatorial lemma}\label{subsec:combinatorial}\hfill

\noindent
The following Lemma~\ref{lem:comb-1} will later be used in the proofs of the local results
Theorem~\ref{thm:quasi-diag}, and Corollary~\ref{cor:projections-that-annihilate}.  We remark that in the infinite
dimensional setting (see~\cite{lechner:2016:factor-SL}), we used that the Rademacher functions tend
to $0$ in a specific way which is related to weak and weak$^*$ convergence (see~\cite[Lemma~4.1,
Lemma~4.2]{lechner:2016:factor-SL}).  Lemma~\ref{lem:comb-1} can be viewed as a quantitative substitute
for those arguments.

There are versions of Lemma~\ref{lem:comb-1} involving spaces other than $H^1$ and $SL^\infty$,
e.g. bi-parameter $H^1$ and bi-parameter BMO (see~\cite[Lemma~4.1]{lechner:mueller:2014}), or mixed
norm Hardy spaces (see~\cite[Lemma~4.1]{lechner:2016-factor-mixed}).  Their origin is the version
for one parameter $H^1$ and one parameter BMO (see~\cite{mueller:1988}; see
also~\cite[Lemma~5.2.4]{mueller:2005}).
\begin{lem}\label{lem:comb-1}
  Let $i \in \mathbb{N}$, $K_0\in \mathcal{D}$, and let $f_j \in SL^\infty$ and
  $g_j \in (SL^\infty)^*$, $1\leq j\leq i$, be such that
  \begin{equation}\label{eq:function-hypothesis}
    \sum_{j=1}^{i} \|f_j\|_{SL^\infty} \leq 1
    \quad\text{and}\quad
    \sum_{j=1}^{i} \|g_j\|_{(SL^\infty)^*} \leq |K_0|.
  \end{equation}
  The local frequency weight $\omega : \mathcal{D}\to [0,\infty)$ is given by
  \begin{equation}\label{eq:local_frequency_weight}
    \omega(K)
    = \sum_{j=1}^{i} |\langle f_j, h_{K} \rangle| + |\langle h_{K}, g_j \rangle|,
    \qquad K\in\mathcal{D}.
  \end{equation}
  Given $\tau > 0$, and $r \in \mathbb{N}_0$ with $2^{-r}\leq |K_0|$, we define the collection of
  dyadic intervals
  \begin{equation*}
    \mathcal{G}(K_0)
    = \big\{ K\in\mathcal{D} :
    K \subset K_0,\ |K|\leq 2^{-r},\
    \omega(K)\leq \tau |K|
    \big\}.
  \end{equation*}
  Moreover, we put
  \begin{equation*}
    \mathcal{G}_k(K_0)
    = \mathcal{G}(K_0)\cap \mathcal{D}_k,
    \qquad k\in\mathbb{N}.
  \end{equation*}
  Then for each $\rho > 0$, there exists an integer $k$ with
  \begin{equation}\label{lem:comb-1:int}
    r \leq k \leq \bigg\lfloor \frac{4}{\rho^2\tau^2} \bigg\rfloor + r
  \end{equation}
  such that
  \begin{equation}\label{lem:comb-1:measure}
    \big| \bigcup \mathcal{G}_k(K_0) \big|
    \geq (1-\rho) |K_0|.
  \end{equation}
\end{lem}
See~Figure~\ref{fig:combinatorial} for a depiction of the collections $\mathcal{G}_k(K_0)$,
$r \leq k \leq \bigg\lfloor \frac{4}{\rho^2\tau^2} \bigg\rfloor + r$.
\begin{figure}[bt]
  \begin{center}
    \includegraphics[scale=0.25]{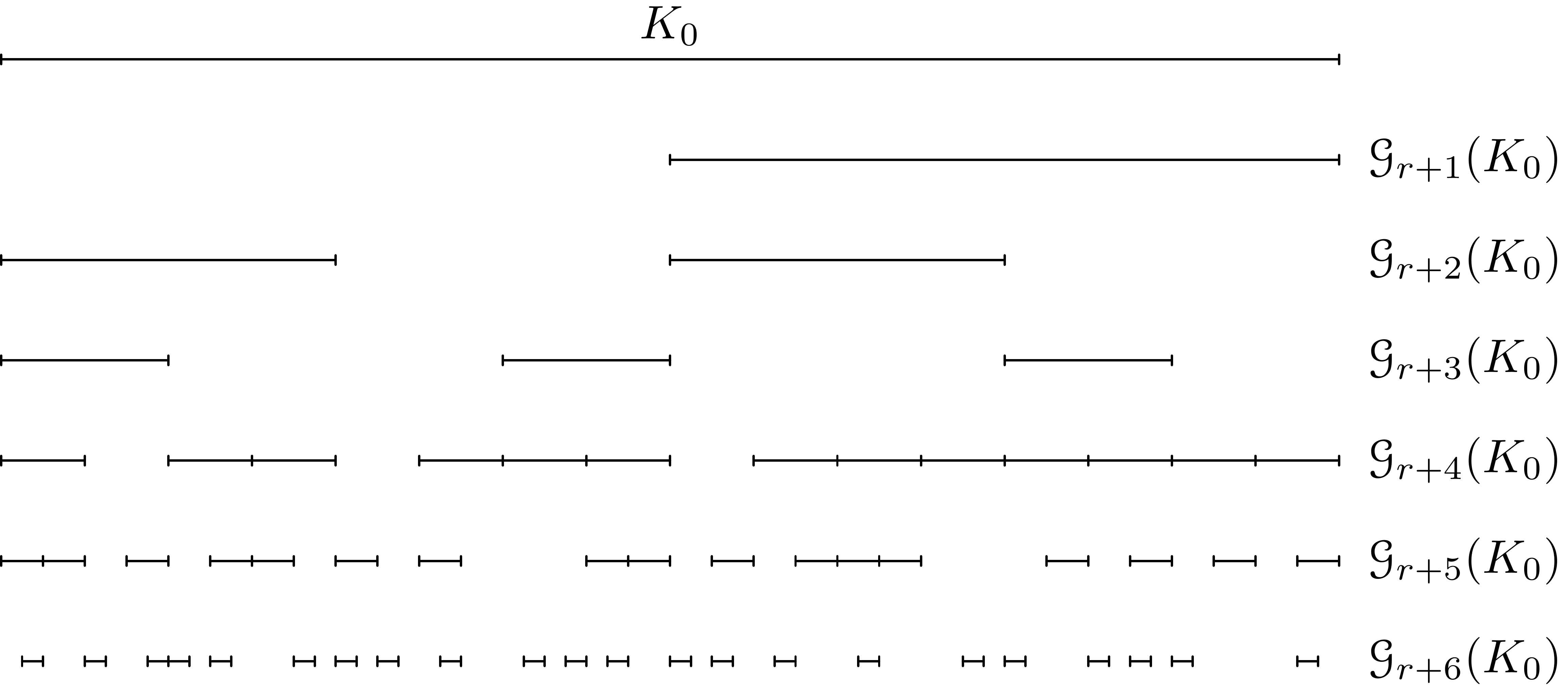}
  \end{center}
  \caption{The interval $K_0$ at the top is partially covered by each of the collections
    $\mathcal{G}_k$, $r+1 \leq k \leq r+6$.  In this case, $\bigcup \mathcal{G}_{r+4}$ is the
    largest in measure of all of those sets.}
  \label{fig:combinatorial}
\end{figure}

\begin{proof}
  The proof of Lemma~\ref{lem:comb-1} is obtained by essentially repeating the argument given
  for~\cite[Lemma~4]{mueller:1988}.  Define
  \begin{equation*}
    \mathcal{H}_k(K_0) = \{ K\in\mathcal{D}_k : K\subset K_0,\ K\notin \mathcal{G}_k(K_0)\},
    \qquad k\in \mathbb{N}_0,
  \end{equation*}
  and put
  \begin{equation*}
    A = \bigg\lfloor \frac{4}{\rho^2\tau^2} \bigg\rfloor + r.
  \end{equation*}
  Assume that the conclusion of Lemma~\ref{lem:comb-1} is not true, i.e.  assume that
  \begin{equation*}
    \big| \bigcup \mathcal{H}_k(K_0) \big|
    > \rho |K_0|,
    \qquad r\leq k \leq A.
  \end{equation*}
  On the one hand, summing these above estimates yields
  \begin{equation}\label{eq:lem:comb-1:1}
    \sum_{k=r}^A \big| \bigcup \mathcal{H}_k(K_0) \big|
    \geq (A-r+1)\, \rho |K_0|.
  \end{equation}
  On the other hand, observe that by definition of $\mathcal{H}_k(K_0)$ and $\mathcal{G}(K_0)$, we
  have
  \begin{equation}\label{eq:lem:comb-1:2}
    \tau \sum_{k=r}^A \big|\bigcup \mathcal{H}_k(K_0)\big|
    \leq \sum_{j=1}^{i} \sum_{k=r}^A \sum_{K \in \mathcal{H}_k}
    |\langle f_j, h_K \rangle| + |\langle h_K, g_j\rangle|.
  \end{equation}
  Now, we rewrite the right hand side of \eqref{eq:lem:comb-1:2} in the following way:
  \begin{equation}\label{eq:lem:comb-1:3}
    \sum_{j=1}^{i} \Big|\Big\langle
    f_j, \sum_{k=r}^A \sum_{K\in\mathcal{H}_k(K_0)} \pm h_K
    \Big\rangle\Big|
    + \Big|\Big\langle
    \sum_{k=r}^A \sum_{K\in\mathcal{H}_k(K_0)} \pm h_K, g_j
    \Big\rangle\Big|.
  \end{equation}
  Note the estimates:
  \begin{subequations}\label{eq:lem:comb-1:4}
    \begin{align}
      \Big\|
      \sum_{k=r}^A \sum_{K\in\mathcal{H}_k(K_0)} \pm h_K
      \Big\|_{H^1}
      & \leq \sqrt{A-r+1}|K_0|,\\
      \Big\|
      \sum_{k=r}^A \sum_{K\in\mathcal{H}_k(K_0)} \pm h_K
      \Big\|_{SL^\infty}
      & \leq \sqrt{A-r+1}.
    \end{align}
  \end{subequations}
  Combining~\eqref{eq:lem:comb-1:2} and~\eqref{eq:lem:comb-1:3} with~\eqref{eq:lem:comb-1:4}, and
  using~\eqref{eq:function-hypothesis} yields
  \begin{align*}
    \tau \sum_{k=r}^A \big|\bigcup \mathcal{H}_k(K_0)\big|
    & \leq \sum_{j=1}^{i} \|f_j\|_{SL^\infty} \Big\|
      \sum_{k=r}^A \sum_{K\in\mathcal{H}_k(K_0)} \pm h_K
      \Big\|_{H^1}\\
    & \qquad + \sum_{j=1}^{i} \|g_j\|_{(SL^\infty)^*}
      \Big\|
      \sum_{k=r}^A \sum_{K\in\mathcal{H}_k(K_0)} \pm h_K
      \Big\|_{SL^\infty}\\
    & \leq 2\sqrt{A-r+1}|K_0|
  \end{align*}
  By~\eqref{eq:lem:comb-1:1} and the above estimate we obtain
  \begin{equation*}
    A \leq \frac{4}{\rho^2\tau^2} + r - 1,
  \end{equation*}
  which contradicts the definition of $A$.
\end{proof}

\subsection{Quantitative diagonalization of operators on
  \bm{$SL^\infty_N$}}\label{subsec:diagonalization}\hfill

\noindent
Here, we will show that any given operator $T$ acting on $SL^\infty_N$ which has large diagonal with
respect to the Haar system, can be almost-diagonalized by a block basis of the Haar system
$(b_I^{(\varepsilon)} : I\in \mathcal{D}^n)$ (see \eqref{eq:block-basis} for the definition of
$b_I^{(\varepsilon)}$) in the space $SL^\infty_N$, that spans a complemented copy of $SL^\infty_n$
(the constants for the norms of the isomorphism and the projection do not depend on $n$; see
Theorem~\ref{thm:quasi-diag}), where the dimensions $N$ and $n$ are quantitatively linked.  This will be
achieved by constructing the block basis $(b_I^{(\varepsilon)} : I\in \mathcal{D}^n)$ so that Jones'
compatibility conditions~\textrefp[J]{enu:j1}--\textrefp[J]{enu:j4} are satisfied.  In order to keep
the diagonal of the operator (with respect to the block new basis
$(b_I^{(\varepsilon)} : I\in \mathcal{D})$) large, we will choose signs appropriately; this
technique was introduced by Andrew in~\cite{andrew:1979}.  The signs in~\cite{andrew:1979} are
selected semi-probabilistically, where in contrast our argument is entirely probabilistic.

From here on, we will regularly identify a dyadic interval $I \in \mathcal{D}$ with the natural
number $\mathcal{O}(I)$ given by
\begin{equation}\label{eq:natural-order}
  \mathcal{O}(I) = 2^n - 1 + k,
  \qquad\text{if $I=[(k-1)2^{-n},k2^{-n})$}.
\end{equation}
Specifically, for $\mathcal{O}(I) = i$ we identify
\begin{equation*}
  \mathcal{B}_I = \mathcal{B}_i
  \qquad\text{and}\qquad
  b_I^{(\varepsilon)} = b_i^{(\varepsilon)}.
\end{equation*}

\begin{thm}\label{thm:quasi-diag}
  Let $n\in\mathbb{N}_0$, $\Gamma,\eta > 0$ and $\delta \geq 0$.  Then there exists an integer
  $N = N(n,\Gamma,\eta)$, such that for any operator $T : SL^\infty_N\to SL^\infty_N$ with
  $\|T\|\leq \Gamma$ and
  \begin{equation*}
    |\langle T h_K, h_K \rangle| \geq \delta |K|,
    \qquad K\in \mathcal{D}^N,
  \end{equation*}
  there exists a finite sequence of collections $(\mathcal{B}_I : I\in \mathcal{D}^n)$ and a finite
  sequence of signs $\varepsilon = (\varepsilon_K \in \{\pm 1\} : K\in \mathcal{B})$, where
  $\mathcal{B} = \bigcup_{I\in\mathcal{D}^n}\mathcal{B}_I$, which generate the block basis of the
  Haar system $(b_I^{(\varepsilon)} : I\in \mathcal{D}^n)$ given by
  \begin{equation*}
    b_I^{(\varepsilon)} = \sum_{K \in \mathcal{B}_I} \varepsilon_K h_K,
    \qquad I\in \mathcal{D}^n,
  \end{equation*}
  such that the following conditions are satisfied:
  \begin{enumerate}[(i)]
  \item\label{enu:thm:quasi-diag-i} $\mathcal{B}_I\subset \mathcal{D}^N$, $I\in\mathcal{D}^n$, and
    $(\mathcal{B}_I : I\in \mathcal{D}^n)$ satisfies Jones' compatibility
    conditions~\textrefp[J]{enu:j1}--\textrefp[J]{enu:j4} with constant $\kappa_J = (1-\eta)^{-1}$.
    
  \item\label{enu:thm:quasi-diag-ii} $(b_I^{(\varepsilon)} : I\in \mathcal{D}^n)$
    almost-diagonalizes $T$ in such way that $T$ has $\delta$-large diagonal with respect to
    $(b_I^{(\varepsilon)} : I\in \mathcal{D}^n)$.  To be more precise, we have the estimates
    \begin{subequations}\label{eq:thm:quasi-diag-ii}
      \begin{align}
        \sum_{j=1}^{i-1} |\langle T b_j^{(\varepsilon)}, b_i^{(\varepsilon)}\rangle|
        + |\langle b_i^{(\varepsilon)}, T^* b_j^{(\varepsilon)}\rangle|
        & \leq \eta 4^{-i} \|b_i^{(\varepsilon)}\|_2^2,
          \label{eq:thm:quasi-diag-ii:a}\\
        \langle T b_i^{(\varepsilon)}, b_i^{(\varepsilon)} \rangle
        & \geq \delta\|b_i^{(\varepsilon)}\|_2^2,
          \label{eq:thm:quasi-diag-ii:b}
      \end{align}
    \end{subequations}
    for all $1\leq i\leq 2^{n+1}-1$.
  \end{enumerate}
\end{thm}

\begin{myproof}
  The proof is divided into the following steps:
  \begin{itemize}
  \item Preparation: setting up the inductive argument;
    
  \item Construction of $\mathcal{B}_{i_0}$: using the combinatorial Lemma~\ref{lem:comb-1} to select
    $\mathcal{B}_{i_0}$;
    
  \item Choosing the signs $\varepsilon_K\in\{\pm 1\}$, $K\in \mathcal{B}_{i_0}$: using a
    probabilistic argument;
    
  \item $(\mathcal{B}_j : j\leq i_0)$ satisfies Jones' compatibility conditions: verifying
    that~\textrefp[J]{enu:j1}--\textrefp[J]{enu:j4} is satisfied with constant
    $\kappa_J = (1-\eta)^{-1}$;
    
  \item $(b_i^{(\varepsilon)} : i\leq i_0)$ almost-diagonalizes $T$: showing
    that~\eqref{eq:thm:quasi-diag-ii:a} is satisfied;
    
  \item Conclusion of the proof: summarizing the previous steps.
  \end{itemize}
  
  \begin{proofstep}[Preparation]\hfill\\
    Let $n\in\mathbb{N}_0$, $\Gamma,\eta > 0$, $\delta \geq 0$, and define the following constants,
    which will be used within the proof:
    \begin{equation}\label{eq:proof:constants:1}
      \begin{aligned}
        \rho_i &= \eta 2^{-i},
        & \tau_{i+1} &= \eta 8^{-i-1} 2^{-m_i}/\Gamma,\\
        m_1 &=0, & m_{i+1} & = m_i + 1 + \Big\lfloor\frac{4}{\rho_{i+1}^2\tau_{i+1}^2}\Big\rfloor, &
        N & = m_{2^{n+1}-1},
      \end{aligned}
    \end{equation}
    for all $1\leq i\leq 2^{n+1}-1$.  Clearly, $N$ depends only on $n,\Gamma$ and $\eta$, i.e.
    $N=N(n,\Gamma,\eta)$.  Finally, let $T : SL^\infty_N\to SL^\infty_N$ be such that
    $\|T\|\leq \Gamma$ and
    \begin{equation*}
      |\langle T h_K, h_K \rangle| \geq \delta |K|,
      \qquad K\in \mathcal{D}^N.
    \end{equation*}

    Before we proceed with the proof, observe first, that by $1$-unconditionality, we can assume
    that
    \begin{equation*}
      \langle T h_K, h_K \rangle \geq \delta |K|,\qquad K\in \mathcal{D}^N.
    \end{equation*}
    Second, given $K \in \mathcal{D}^N$, we write
    \begin{subequations}\label{eq:decomp}
      \begin{equation}
        T h_K = \alpha_K h_K + r_K,
      \end{equation}
      where
      \begin{equation}
        \alpha_K = \frac{\langle T h_K, h_K \rangle}{|K|}
        \quad\text{and}\quad
        r_K = \sum_{\substack{L\in\mathcal{D}^N\\L\neq K}}
        \frac{\langle T h_K, h_L \rangle}{|L|} h_L.
      \end{equation}
    \end{subequations}
    Thirdly, note the estimate
    \begin{equation}\label{eq:a-estimate}
      \delta \leq \alpha_K \leq \|T\|,
      \qquad K\in\mathcal{D}^N.
    \end{equation}
  \end{proofstep}

  We will now inductively define the block basis $(b_I^{(\varepsilon)} : I\in\mathcal{D}^n)$.  To
  begin the induction, we simply put
  \begin{equation*}
    \mathcal{B}_1 = \mathcal{B}_{[0,1)} = \{[0,1)\}
    \qquad\text{and}\qquad
    b_1^{(\varepsilon)} = b_{[0,1)}^{(\varepsilon)} = h_{[0,1)}.
  \end{equation*}
  and note that $\mathcal{B}_1\subset\mathcal{D}^{m_1}$ by~\eqref{eq:proof:constants:1}.

  For the inductive step, let $i_0\geq 2$ and assume that
  \begin{itemize}
  \item we have already chosen finite collections $\mathcal{B}_j$ with
    $\mathcal{B}_j\subset\mathcal{D}^{m_j}\setminus\mathcal{D}^{m_{j-1}}$, $2\leq j \leq i_0-1$;

  \item the finite sequence of collections $(\mathcal{B}_j : 1 \leq j \leq i_0-1)$
    satisfies~\textrefp[J]{enu:j1}--\textrefp[J]{enu:j4} with constant
    $\kappa_J = (1 - 3\sum_{j=1}^{i_0-1}\rho_j)^{-1}$;

  \item we made a suitable choice of signs
    $\varepsilon = (\varepsilon_K \in \{\pm 1\} : K\in \bigcup_{j=1}^{i_0-1}\mathcal{B}_j)$;

  \item the block basis elements $b_j^{(\varepsilon)}$ have been defined by
    \begin{equation*}
      b_j^{(\varepsilon)} = \sum_{K\in \mathcal{B}_j} \varepsilon_K h_K,
      \qquad 1\leq j\leq i_0-1,
    \end{equation*}
    and satisfy~\eqref{eq:thm:quasi-diag-ii} for all $1\leq i\leq i_0-1$.
  \end{itemize}

  In the following step, we will choose a finite collection
  $\mathcal{B}_{i_0}\subset\mathcal{D}^{m_{i_0}}\setminus\mathcal{D}^{m_{i_0-1}}$, select signs
  $\varepsilon = (\varepsilon_K\in \{\pm 1\} : K\in \mathcal{B}_{i_0})$, so that if we put
  \begin{equation}\label{eq:induction-properties}
    b_{i_0}^{(\varepsilon)}
    = \sum_{K\in \mathcal{B}_{i_0}} \varepsilon_K h_K,
  \end{equation}
  $(b_j : 1\leq j \leq i_0)$ satisfies~\eqref{eq:thm:quasi-diag-ii} for all $1\leq i\leq i_0$, and
  $(\mathcal{B}_j : 1\leq j\leq i_0)$ satisfies~\textrefp[J]{enu:j1}--\textrefp[J]{enu:j4} with
  constant $\kappa_J = (1 - 3\sum_{j=1}^{i_0}\rho_j)^{-1}$.

  \begin{proofstep}[Construction of $\mathcal{B}_{i_0}$]\hfill\\
    Let $I_0\in \mathcal{D}$ be such that $\mathcal{O}(I_0) = i_0$.  Let $\widetilde I_0$ denote the
    unique dyadic interval with $\widetilde I_0 \supset I_0$ and $|\widetilde I_0| = 2 |I_0|$.
    Furthermore, for every dyadic interval $K_0$, we denote its left half by $K_0^\ell$ and its
    right half by $K_0^r$.  Following the construction principle of
    Gamlen-Gaudet~\cite{gamlen:gaudet:1973}, we define the collections of dyadic intervals
    \begin{equation*}
      \mathcal{B}_{\widetilde I_0}^\ell
      = \{ K_0^\ell : K_0\in \mathcal{B}_{\widetilde I_0} \}
      \qquad\text{and}\qquad
      \mathcal{B}_{\widetilde I_0}^r
      = \{ K_0^r : K_0\in \mathcal{B}_{\widetilde I_0} \}.
    \end{equation*}
    By induction hypothesis we know that
    $\mathcal{B}_{\widetilde I_0} = \mathcal{B}_{\mathcal{O}(\widetilde I_0)} \subset
    \mathcal{D}^{m_{\mathcal{O}(\widetilde I_0)}} \setminus \mathcal{D}^{m_{\mathcal{O}(\widetilde
        I_0)-1}}$, thus, since $(m_i)_{i=1}^{2^{n+1}-1}$ is a finite increasing sequence
    by~\eqref{eq:proof:constants:1}, and $\mathcal{O}(\widetilde I_0) \leq i_0-1$, we obtain
    \begin{equation}\label{eq:proof:past-frequencies}
      \min\{ |K_1| : K_1\in \mathcal{B}_{\widetilde I_0}^\ell\cup \mathcal{B}_{\widetilde I_0}^r \}
      \geq 2^{-m_{\widetilde I_0}-1}
      \geq 2^{-m_{i_0-1}-1}.
    \end{equation}
    Now, define the functions
    \begin{equation}\label{eq:past-functions}
      f_j
      = \frac{T b_j^{(\varepsilon)}}{2^j 2^{m_{i_0-1}+1}\Gamma}
      \qquad\text{and}\qquad
      g_j
      =  \frac{T^* b_j^{(\varepsilon)}}{2^j 2^{m_{i_0-1}+1}\Gamma},
      \qquad 1\leq j\leq i_0-1,
    \end{equation}
    and note that we have the estimates
    \begin{equation*}
      \sum_{j=1}^{i_0-1}\|f_j\|_{SL^\infty}
      \leq 1
      \quad\text{and}\quad
      \sum_{j=1}^{i_0-1} \|g_j\|_{(SL^\infty)^*}
      \leq 2^{-m_{i_0-1}-1},
      \quad 1\leq j\leq i_0-1.
    \end{equation*}
    The local frequency weight $\omega_{i_0} : \mathcal{D}\to [0,\infty)$ is given by
    \begin{equation}\label{eq:proof:frequency-weight}
      \omega_{i_0-1}(K)
      = \sum_{j=1}^{i_0-1} |\langle f_j, h_{K} \rangle| + |\langle h_K, g_j \rangle|,
      \qquad K\in\mathcal{D}.
    \end{equation}
    By Lemma~\ref{lem:comb-1} and~\eqref{eq:proof:past-frequencies}, we find for each
    $K_1\in \mathcal{B}_{\widetilde I_0}^\ell\cup \mathcal{B}_{\widetilde I_0}^r$ a collection of
    pairwise disjoint dyadic intervals $\mathcal{Z}_{i_0}(K_1)$ such that
    \begin{subequations}\label{eq:proof:collection:properties}
      \begin{equation}\label{eq:proof:collection:properties:a}
        \big|\bigcup \mathcal{Z}_{i_0}(K_1)\big|
        \geq (1-\rho_{i_0})|K_1|,
        \qquad\omega_{i_0-1}(K)
        \leq \tau_{i_0} |K|,\ K\in\mathcal{Z}_{i_0}(K_1).
      \end{equation}
      Recall that by~\eqref{eq:proof:constants:1} we have
      $m_{i_0} = m_{i_0-1} + 1 + \big\lfloor\frac{4}{\rho_{i_0}^2\tau_{i_0}^2}\big\rfloor$,
      hence~\eqref{eq:proof:past-frequencies} yields
      \begin{equation}\label{eq:proof:collection:properties:b}
        \min\{ |K| : K\in\mathcal{Z}_{i_0}(K_1) \}
        \geq |K_1| 2^{-\big\lfloor\frac{4}{\rho_{i_0}^2\tau_{i_0}^2}\big\rfloor}
        \geq 2^{-m_{i_0}}.
      \end{equation}
    \end{subequations}
    If $I_0$ is the \emph{left half} of $\widetilde I_0$, we put
    \begin{subequations}\label{eq:intervals}
      \begin{equation}\label{eq:intervals:a}
        \mathcal{B}_{i_0}
        = \mathcal{B}_{I_0}
        = \bigcup \big\{ \mathcal{Z}_{i_0}(K_1) : K_1\in \mathcal{B}_{\widetilde I_0}^\ell \big\},
      \end{equation}
      and if $I_0$ is the \emph{right half} of $\widetilde I_0$, we define
      \begin{equation}\label{eq:intervals:b}
        \mathcal{B}_{i_0}
        = \mathcal{B}_{I_0}
        = \bigcup \big\{ \mathcal{Z}_{i_0}(K_1) : K_1\in \mathcal{B}_{\widetilde I_0}^r \big\}.
      \end{equation}
      See Figure~\ref{fig:gamlen-gaudet} for a depiction of the collection $\mathcal{B}_{i_0}$.
      \begin{figure}[bt]
        \begin{center}
          \includegraphics{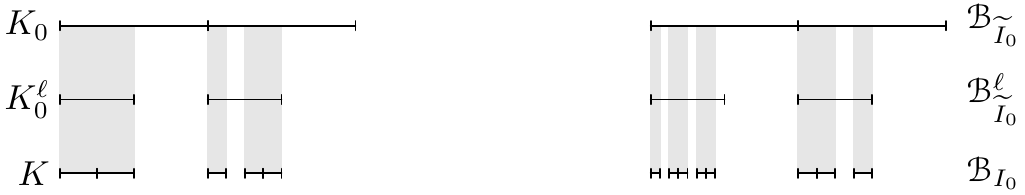}
        \end{center}
        \caption{The picture shows the construction of $\mathcal{B}_{I_0}$, if $I_0$ is the left
          half of $\widetilde I_0$.  The large dyadic intervals $K_0$ on top form the set
          $\mathcal{B}_{\widetilde I_0}$.  The medium sized dyadic intervals $K_0^\ell$ denote the
          left half of the $K_0$, and the set $\mathcal{B}_{\widetilde I_0}^\ell$ is the collection
          of all the $K_0^\ell$.  The small intervals $K$ at the bottom (which are selected with
          Lemma~\ref{lem:comb-1}) form the almost-cover of each of the intervals in
          $\bigcup \mathcal{B}_{\widetilde I_0}^\ell$.  The collection of all those $K$ is denoted
          by $\mathcal{B}_{I_0}$.}
        \label{fig:gamlen-gaudet}
      \end{figure}
    \end{subequations}
    By~\eqref{eq:proof:collection:properties:b} and~\eqref{eq:intervals}, we have the inclusion
    \begin{equation}\label{eq:intervals:frequency:estimate}
      \mathcal{B}_{i_0}\subset \mathcal{D}^{m_{i_0}}.
    \end{equation}
    In either of the cases~\eqref{eq:intervals:a} and~\eqref{eq:intervals:b}, we put (with a slight
    abuse of notation)
    \begin{equation}\label{eq:block-basis-candidate}
      b_{i_0}^{(\varepsilon)}
      = b_{I_0}^{(\varepsilon)}
      = \sum_{K\in \mathcal{B}_{I_0}} \varepsilon_K h_K,
    \end{equation}
    for all choices of signs $\varepsilon = (\varepsilon_K\in\{\pm 1\} : K\in \mathcal{B}_{I_0})$.
  \end{proofstep}

  \begin{proofstep}[Choosing the signs $\varepsilon_K\in\{\pm 1\}$, $K\in \mathcal{B}_{i_0}$]\hfill\\
    Continuing with the proof, we obtain from~\eqref{eq:decomp} that for all choices of signs
    $\varepsilon = (\varepsilon_K\in\{\pm 1\} : K\in \mathcal{B}_{I_0})$
    \begin{equation}\label{eq:block-basis-candidate:identity}
      T b_{i_0}^{(\varepsilon)}
      = \sum_{K\in \mathcal{B}_{i_0}} \varepsilon_K \alpha_K h_K + R_{i_0}^{(\varepsilon)},
    \end{equation}
    where
    \begin{equation}\label{eq:block-basis-candidate:rest}
      R_{i_0}^{(\varepsilon)} = \sum_{K\in \mathcal{B}_{i_0}} \varepsilon_K r_K.
    \end{equation}
    Now, put
    \begin{equation*}
      X_{i_0}(\varepsilon) = \langle R_{i_0}^{(\varepsilon)}, b_{i_0}^{(\varepsilon)} \rangle,
      \qquad \varepsilon = (\varepsilon_K\in \{\pm 1\} : K\in \mathcal{B}_{i_0}),
    \end{equation*}
    and observe that by~\eqref{eq:block-basis-candidate:identity} and~\eqref{eq:a-estimate} we
    obtain
    \begin{equation}\label{eq:diagonal-estimate:1}
      \langle T b_{i_0}^{(\varepsilon)}, b_{i_0}^{(\varepsilon)}\rangle
      \geq \delta \|b_{i_0}^{(\varepsilon)}\|_2^2 + X_{i_0}(\varepsilon),
      \qquad \varepsilon = (\varepsilon_K\in \{\pm 1\} : K\in \mathcal{B}_{i_0}).
    \end{equation}
    By~\eqref{eq:decomp}, we have that $\langle r_K, h_K \rangle = 0$, hence
    \begin{equation*}
      X_{i_0}(\varepsilon)
      = \sum_{\substack{K_0, K_1\in \mathcal{B}_{i_0}\\K_0\neq K_1}} \varepsilon_{K_0} \varepsilon_{K_1}
      \langle r_{K_0}, h_{K_1} \rangle,
      \qquad \varepsilon = (\varepsilon_K\in \{\pm 1\} : K\in \mathcal{B}_{i_0}).
    \end{equation*}
    Now, let $\cond_\varepsilon$ denote the average over all possible choices of signs
    $\varepsilon = (\varepsilon_K \in \{\pm 1\} : K \in \mathcal{B}_{i_0})$.  If $K_0\neq K_1$, then
    $\cond_\varepsilon \varepsilon_{K_0} \varepsilon_{K_1} = 0$; therefore
    \begin{equation*}
      \cond_\varepsilon X_{i_0} = 0.
    \end{equation*}
    Taking the average $\cond_\varepsilon$ in~\eqref{eq:diagonal-estimate:1} we obtain
    \begin{equation}\label{eq:diagonal-estimate:2}
      \cond_\varepsilon \langle T b_{i_0}^{(\varepsilon)}, b_{i_0}^{(\varepsilon)}\rangle
      \geq \delta \cond_\varepsilon \|b_{i_0}^{(\varepsilon)}\|_2^2.
    \end{equation}
    Hence, by~\eqref{eq:diagonal-estimate:2} there is an
    $\varepsilon = (\varepsilon_K\in \{\pm 1\} : K\in \mathcal{B}_{i_0})$ such that
    \begin{equation}\label{eq:diagonal_estimate}
      \langle T b_{i_0}^{(\varepsilon)}, b_{i_0}^{(\varepsilon)} \rangle
      \geq \delta \|b_{i_0}^{(\varepsilon)}\|_2^2.
    \end{equation}
    This concludes the constructive part of the inductive step.  Next, we will demonstrate that our
    construction has the properties claimed in the theorem.
  \end{proofstep}

  \begin{proofstep}[$(\mathcal{B}_j : 1\leq j\leq i_0)$ satisfies Jones' compatibility
    conditions]\hfill\\
    Note that by our induction hypothesis, \eqref{eq:natural-order}
    and~\eqref{eq:intervals:frequency:estimate}, we obtain that
    $\mathcal{B}_j\subset \mathcal{D}^N$, $1\leq j \leq i_0$.  We will now show inductively that the
    finite sequence of collections of dyadic intervals $(B_j : 1\leq j\leq i_0)$
    satisfies~\textrefp[J]{enu:j1}--\textrefp[J]{enu:j4} with constant
    $\kappa_J = (1 - \sum_{j=1}^{i_0}\rho_j)^{-1}$.  By our choice of $\mathcal{Z}_{i_0}(K_1)$
    in~\eqref{eq:proof:collection:properties}, it should be clear that the
    properties~\textrefp[J]{enu:j1}--\textrefp[J]{enu:j3} are satisfied.  We will now prove
    that~\textrefp[J]{enu:j4} is satisfied, as well.

    To this end, let us record that by induction hypothesis, we have the estimate
    \begin{equation}\label{eq:proof:J4:ind-hypo}
      \frac{|K\cap B_{I_0}|}{|K|}
      \geq \Big(1 - \sum_{j=1}^{i_0-1}\rho_j\Big) \frac{|B_{I_0}|}{|B_I|},
    \end{equation}
    for all $I_0,I\in\mathcal{D}^n$ with $\mathcal{O}(I_0)\leq i_0-1$, $I_0\subset I$ and
    $K\in \mathcal{B}_I$.  Now, let $I_0,I\in\mathcal{D}^n$ be such that $\mathcal{O}(I_0)=i_0$,
    $I\supsetneq I_0$ and let $K\in \mathcal{B}_I$.  First, note that~\textrefp[J]{enu:j3},
    \textrefp[J]{enu:j2} and Lemma~\ref{lem:projection-simple}~\eqref{enu:projection-simple:3} give us
    \begin{equation*}
      |K\cap B_{I_0}|
      = |K\cap B_{I_0}\cap B_{\widetilde I_0}|
      = \sum_{\substack{L\in\mathcal{B}_{\widetilde I_0}\\L\subset K}}
      |L\cap B_{I_0}|
      = \sum_{\substack{L\in\mathcal{B}_{\widetilde I_0}\\L\subset K}}
      |L^\ell\cap B_{I_0}| + |L^r\cap B_{I_0}|.
    \end{equation*}
    Recall that $L^\ell$ denotes the left half of $L$, and that $L^r$ denotes the right half of $L$.
    Considering our choice for $\mathcal{Z}_{I_0}(L^\ell)$ and $\mathcal{Z}_{I_0}(L^r)$
    in~\eqref{eq:proof:collection:properties:a}, and for $\mathcal{B}_{I_0}$
    in~\eqref{eq:intervals}, we find that
    \begin{equation*}
      |L^\ell\cap B_{I_0}|\geq (1-\rho_{I_0})|L^\ell|
      \quad\text{and}\quad
      |L^r\cap B_{I_0}|\geq (1-\rho_{I_0})|L^r|,
      \qquad L\in\mathcal{B}_{I_0}.
    \end{equation*}
    Now observe that by~\eqref{eq:intervals}, $B_{I_0}\cap L^r = \emptyset$, if $I_0$ is the left
    half of $\widetilde I_0$, and that $B_{I_0}\cap L^\ell = \emptyset$, if $I_0$ is the right half
    of $\widetilde I_0$.  Combining everything after~\eqref{eq:proof:J4:ind-hypo} with
    Lemma~\ref{lem:projection-simple}~\eqref{enu:projection-simple:3} yields
    \begin{equation}\label{eq:proof:J4:1}
      |K\cap B_{I_0}|
      \geq \frac{1}{2}(1-\rho_{I_0})
      \sum_{\substack{L\in\mathcal{B}_{\widetilde I_0}\\L\subset K}} |L|
      = \frac{1}{2}(1-\rho_{I_0}) |K\cap B_{\widetilde I_0}|.
    \end{equation}
    Similar considerations give us
    \begin{equation}\label{eq:proof:J4:2}
      |B_{\widetilde I_0}|
      = \sum_{L\in\mathcal{B}_{\widetilde I_0}} |L^\ell|
      = \sum_{L\in\mathcal{B}_{\widetilde I_0}} |L^r|
      \geq 2 |B_{I_0}|.
    \end{equation}

    Note that by~\eqref{eq:natural-order} $\mathcal{O}(\widetilde I_0)\leq \mathcal{O}(I_0)-1$.
    Hence, the estimate~\eqref{eq:proof:J4:1}, our induction hypothesis~\eqref{eq:proof:J4:ind-hypo}
    and \eqref{eq:proof:J4:2} yield
    \begin{equation}\label{eq:proof:J4:3}
      \begin{aligned}
        |K\cap B_{I_0}| & \geq \frac{1}{2}(1-\rho_{I_0})\Big(1 - \sum_{j=1}^{i_0-1}\rho_j\Big) |K|
        \frac{|B_{\widetilde I_0}|}{|B_I|}\\
        & \geq \Big(1 - \sum_{j=1}^{i_0}\rho_j\Big) |K| \frac{|B_{I_0}|}{|B_I|}.
      \end{aligned}
    \end{equation}
    Thus, we proved that~\eqref{eq:proof:J4:ind-hypo} holds true for $i_0$ instead of $i_0-1$.
  \end{proofstep}

  \begin{proofstep}[$(b_j^{(\varepsilon)} : 1\leq j\leq i_0)$ almost-diagonalizes $T$]\hfill\\
    By~\eqref{eq:past-functions} and~\eqref{eq:block-basis-candidate}, we obtain
    \begin{equation*}
      \sum_{j=1}^{i_0-1} |\langle T b_j^{(\varepsilon)}, b_{i_0}^{(\varepsilon)}\rangle|
      + |\langle b_{i_0}^{(\varepsilon)}, T^* b_j^{(\varepsilon)}\rangle|
      \leq 2^{i_0}2^{m_{i_0-1}}\Gamma\sum_{K\in\mathcal{B}_{i_0}} \sum_{j=1}^{i_0-1}
      |\big\langle f_j, h_K\big\rangle|
      + |\big\langle h_K, g_j\big\rangle|.
    \end{equation*}
    Recall that by~\eqref{eq:proof:constants:1} $\tau_{i_0} = \eta 8^{-i_0} 2^{-m_{i_0-1}}/\Gamma$,
    thus, the above inequality, \eqref{eq:proof:frequency-weight},
    \eqref{eq:proof:collection:properties:a}, \eqref{eq:intervals}
    and~\eqref{eq:block-basis-candidate} yield
    \begin{equation}\label{eq:proof:off-diagonal_estimate}
      \sum_{j=1}^{i_0-1} |\langle T b_j^{(\varepsilon)}, b_{i_0}^{(\varepsilon)}\rangle|
      + |\langle b_{i_0}^{(\varepsilon)}, T^* b_j^{(\varepsilon)}\rangle|
      \leq 2^{i_0}2^{m_{i_0-1}}\Gamma \sum_{K\in\mathcal{B}_{i_0}} \tau_{i_0} |K|
      = \eta 4^{-i_0} \|b_{i_0}^{(\varepsilon)}\|_2^2.
    \end{equation}
    Hence, \eqref{eq:proof:off-diagonal_estimate} combined with~\eqref{eq:diagonal_estimate} shows
    that~\eqref{eq:thm:quasi-diag-ii} is true for all $1\leq i\leq i_0$.
  \end{proofstep}

  \begin{proofstep}[Conclusion of the proof]\hfill\\
    Thus far, we proved the following:
    \begin{itemize}
    \item we chose finite collections $\mathcal{B}_j$ with
      $\mathcal{B}_j\subset\mathcal{D}^{m_j}\setminus\mathcal{D}^{m_{j-1}}$, $2\leq j \leq i_0$ and
      $\mathcal{B}_1\subset \mathcal{D}^{m_1}$;

    \item the finite sequence of collections $(\mathcal{B}_j : 1 \leq j \leq i_0)$
      satisfies~\textrefp[J]{enu:j1}--\textrefp[J]{enu:j4} with constant
      $\kappa_J = (1 - \sum_{j=1}^{i_0}\rho_j)^{-1}$;

    \item we chose signs
      $\varepsilon = (\varepsilon_K \in \{\pm 1\} : K\in \bigcup_{j=1}^{i_0}\mathcal{B}_j)$;

    \item the block basis elements $b_j^{(\varepsilon)}$ have been defined by
      \begin{equation*}
        b_j^{(\varepsilon)} = \sum_{K\in \mathcal{B}_j} \varepsilon_K h_K,
        \qquad 1\leq j\leq i_0,
      \end{equation*}
      and satisfy~\eqref{eq:thm:quasi-diag-ii} for all $1\leq i\leq i_0$.
    \end{itemize}
    We conclude the proof by stopping the induction process after $2^{n+1}-1$ steps and considering
    the definition of the constants in~\eqref{eq:proof:constants:1}.
  \end{proofstep}
\end{myproof}

\begin{rem}\label{rem:quasi-diag}
  We note the following:
  \begin{enumerate}[(i)]
  \item\label{enu:rem:quasi-diag:i} Recall that in~\eqref{eq:proof:constants:1} we defined
    $\rho_I = \rho_i = \eta 2^{-i}$, $1\leq i\leq 2^{n+1}-1$, whenever $\mathcal{O}(I)=i$.  Now,
    observe that by summing~\eqref{eq:proof:J4:1} over all $K\in\mathcal{B}_I$, we obtain
    by~\textrefp[J]{enu:j2} and~\textrefp[J]{enu:j3} that
    \begin{equation*}
      |B_I|
      \geq \frac{1}{2}(1-\rho_I) |B_{\widetilde I}|,
      \qquad I\in\mathcal{D}^n\setminus\{[0,1)\}.
    \end{equation*}
    Recall that we chose $\mathcal{B}_{[0,1)}=\{[0,1)\}$, and note that iterating the latter
    inequality yields
    \begin{equation}\label{eq:rem:quasi-diag:i}
      |B_I|
      \geq |I| \prod_{J\in\mathcal{D}^n}(1-\rho_J)
      \geq |I| (1-\eta),
      \qquad I\in\mathcal{D}^n.
    \end{equation}
    Furthermore, it is a simple observation that we have the estimate $|B_I|\leq |I|$,
    $I\in\mathcal{D}^n$.
    
  \item\label{enu:rem:quasi-diag:ii} If $\delta = 0$, we can choose the signs arbitrarily; in
    particular, we can choose $\varepsilon_K = 1$, $K\in\mathcal{D}^N$, and in that case
    $b_j^{(\varepsilon)} = b_j$ (see~\eqref{eq:thm:projection:block-basis}).
  \end{enumerate}
\end{rem}

\subsection{Almost-annihilating finite dimensional subspaces of
  \bm{$SL^\infty_N$}}\label{subsec:almost-annihil}\hfill

\noindent
When using Bourgain's localization method, one eventually needs to pass from the local factorization
results to factorization results on the direct sum of the finite dimensional spaces (which in our
case are $SL^\infty_n$, $n\in\mathbb{N}_0$).  One of the ingredients is showing that in a large
enough space, any finite dimensional space can be almost-annihilated by a bounded projection that
has a large image, which goes back to~\cite[Lemma~1]{bourgain:1983}.  The following
Definition~\ref{dfn:property-pafds} is merely an abstract version for sequences of finite dimensional Banach
spaces of the corresponding Lemma in~\cite[Lemma~2]{bourgain:1983}.

In~\cite{lechner:2016-factor-mixed}, the following notion was introduced.
\begin{dfn}\label{dfn:property-pafds}
  We say that a non-decreasing sequence of finite dimensional Banach spaces
  $(X_n)_{n\in \mathbb{N}}$ with $\sup_n \dim X_n = \infty$ has the \emph{property that projections
    almost annihilate finite dimensional subspaces with constant $C_P > 0$}, if the following
  conditions are satisfied:

  For all $n,d\in \mathbb{N}$ and $\eta > 0$ there exists an integer $N=N(n,d,\eta)$ such that for
  any $d$-dimensional subspace $F\subset X_N$ there exists a bounded projection $Q : X_N\to X_N$ and
  an isomorphism $R : X_n\to Q(X_N)$ such that
  \begin{enumerate}[(i)]
  \item $\|Q\| \leq C_P$,
  \item $\|R\|, \|R^{-1}\| \leq C_P$,
  \item $\|Q x\| \leq \eta \|x\|$, for all $x\in F$.
  \end{enumerate}
\end{dfn}

In the following Corollary~\ref{cor:projections-that-annihilate} we establish that
$(SL^\infty_n)_{n\in\mathbb{N}_0}$ has the property that projections almost annihilate finite
dimensional subspaces, which is a crucial ingredient in the proof of Theorem~\ref{thm:primary}.
\begin{cor}\label{cor:projections-that-annihilate}
  Let $n,d \in \mathbb{N}_0$ and $\eta > 0$.  Then there exists an integer $N=N(n,d,\eta)$ such that
  for any $d$-dimensional subspace $F\subset SL^\infty_N$ there exists a block basis
  $(b_I : I\in \mathcal{D}^n)$ satisfying the following conditions:
  \begin{enumerate}[(i)]
  \item\label{enu:cor:projections-that-annihilate-i} $\mathcal B_I \subset \mathcal{D}^N$, for all
    $I \in \mathcal{D}^n$.
    
  \item\label{enu:cor:projections-that-annihilate-ii} For every finite sequence of scalars
    $(a_I : I \in \mathcal{D}^n)$, we have that
    \begin{equation}\label{eq:cor:projections-that-annihilate:ii}
      (1+\eta)^{-1} \Big\| \sum_{I\in \mathcal{D}^n} a_I h_I \Big\|_{SL^\infty}
      \leq \Big\| \sum_{I\in \mathcal{D}^n} a_I b_I \Big\|_{SL^\infty}
      \leq (1+\eta) \Big\| \sum_{I\in \mathcal{D}^n} a_I h_I \Big\|_{SL^\infty}.
    \end{equation}
    
  \item\label{enu:cor:projections-that-annihilate-iii} The orthogonal projection
    $Q : SL^\infty_N \rightarrow SL^\infty_N$ given by
    \begin{equation*}
      Q f = \sum_{I\in \mathcal{D}^n}
      \frac{\langle f, b_I \rangle}{\|b_I\|_2^2}\,
      b_I
    \end{equation*}
    satisfies the estimates
    \begin{equation}\label{eq:cor:projections-that-annihilate:iii}
      \begin{aligned}
        \|Q f\|_{SL^\infty} &\leq (1+\eta) \|f\|_{SL^\infty},
        &f &\in SL^\infty_N,\\
        \| Q f \|_{SL^\infty} &\leq \eta\, \|f\|_{SL^\infty}, &f &\in F.
      \end{aligned}
    \end{equation}
  \end{enumerate}
\end{cor}

\begin{proof}
  Corollary~\ref{cor:projections-that-annihilate} is obtained by modifying the proof of Theorem~\ref{thm:quasi-diag}
  in the following way: We choose a finite $\eta/2$-net $\{f_1,\ldots,f_{k_0}\}$ of the compact unit
  ball of the $d$-dimensional subspace $F$, and we use the frequency weight
  $\omega(K) = \sum_{k=1}^{k_0} |\langle f_k, h_K \rangle|$ in all stages of the proof instead of
  the ones defined in~\eqref{eq:proof:frequency-weight}.
\end{proof}

\begin{rem}\label{rem:property-pafds}
  Corollary~\ref{cor:projections-that-annihilate} implies that for any $\eta > 0$, the sequence of finite
  dimensional Banach spaces $(SL^\infty_n)_{n\in\mathbb{N}_0}$ has the property that projections
  almost annihilate finite dimensional subspaces with constant $1+\eta$.
\end{rem}

\section{Local factorization of the identity operator on $SL^\infty_N$ through operators with large
  diagonal}\label{sec:local-factor}

\noindent
Here, we will prove the local factorization results Theorem~\ref{thm:local-factor} and
Theorem~\ref{thm:local-primary}. The key ingredients are the finite dimensional quantitative
almost-diagonalization Theorem~\ref{thm:quasi-diag} and the projection Theorem~\ref{thm:projection}.

\subsection{Proof of Theorem~\ref{thm:local-factor}}\label{subsec:local-factor}\hfill

\noindent
The basic pattern of the following proof was recently employed in~\cite{lechner:mueller:2014,
  laustsen:lechner:mueller:2015,lechner:2016-factor-mixed,lechner:2016:factor-SL}.

Let $n\in\mathbb{N}_0$, $\Gamma,\eta > 0$, $\delta > 0$, and let $\eta_1 = \eta(n,\delta,\eta)$
denote the largest positive constant such that
\begin{equation}\label{eq:proof:local-factor:eta_1}
  0 < \eta_1 \leq 1/2,
  \qquad
  \frac{\eta_12^{n+1}}{\delta} \leq 1/2,
  \qquad
  \frac{1}{1 - \frac{\eta_12^{n+2}}{\delta}} \leq 1 + \eta.
\end{equation}
Thus, by Theorem~\ref{thm:quasi-diag}, there exists an integer
$N=N(n,\Gamma,\eta_1)=N(n,\Gamma,\eta,\delta)$, such that for any operator
$T : SL^\infty_N\to SL^\infty_N$ with $\|T\|\leq \Gamma$ and
\begin{equation*}
  |\langle T h_I, h_I \rangle| \geq \delta |I|,
  \qquad I\in \mathcal{D}^N,
\end{equation*}
there exists a finite sequence of collections $(\mathcal{B}_I : I\in \mathcal{D}^n)$ and a finite
sequence of signs $\varepsilon = (\varepsilon_K \in \{\pm 1\} : K\in \mathcal{B})$, where
$\mathcal{B} = \bigcup_{I\in\mathcal{D}^n}\mathcal{B}_I$, which generate the block basis of the Haar
system $(b_I^{(\varepsilon)} : I\in \mathcal{D}^n)$ given by
\begin{equation*}
  b_I^{(\varepsilon)} = \sum_{K \in \mathcal{B}_I} \varepsilon_K h_K,
  \qquad I\in \mathcal{D},
\end{equation*}
such that the following conditions are satisfied:
\begin{subequations}\label{eq:proof:local-factor:ii}
  \begin{enumerate}[(i)]
  \item\label{enu:proof:local-factor:quasi-diag-i} $\mathcal{B}_I\subset\mathcal{D}^N$ whenever
    $I\in\mathcal{D}^n$, and $(\mathcal{B}_I : I\in \mathcal{D}^n)$ satisfies Jones' compatibility
    conditions~\textrefp[J]{enu:j1}--\textrefp[J]{enu:j4} with constant
    $\kappa_J = (1-\eta_1)^{-1}$.
    
  \item\label{enu:proof:local-factor:quasi-diag-ii} For all $1\leq i\leq 2^{n+1}-1$, we have the
    estimates
    \begin{align}
      \sum_{j=1}^{i-1} |\langle T b_j^{(\varepsilon)}, b_i^{(\varepsilon)}\rangle|
      + |\langle b_i^{(\varepsilon)}, T^* b_j^{(\varepsilon)}\rangle|
      & \leq \eta_1 4^{-i} \|b_i^{(\varepsilon)}\|_2^2,
        \label{eq:proof:local-factor:ii:a}\\
      \langle T b_i^{(\varepsilon)}, b_i^{(\varepsilon)} \rangle
      & \geq \delta\|b_i^{(\varepsilon)}\|_2^2.
        \label{eq:proof:local-factor:ii:b}
    \end{align}
    Recall that $\mathcal{B}_i=\mathcal{B}_I$ and $b_i^{(\varepsilon)}=b_I^{(\varepsilon)}$,
    whenever $\mathcal{O}(I)=i$.
  \end{enumerate}
  Moreover, by~\textrefp[J]{enu:j2} and~\eqref{eq:rem:quasi-diag:i} in Remark~\ref{rem:quasi-diag}, we
  have the estimate
  \begin{equation}\label{eq:proof:local-factor:ii:c}
    \|b_I^{(\varepsilon)}\|_2^2
    = |B_I|
    \geq (1-\eta_1) |I|
    \geq (1-\eta_1) 2^{-n},
    \qquad I\in\mathcal{D}^n.
  \end{equation}
\end{subequations}

Since $(\mathcal{B}_I : I\in \mathcal{D}^n)$ satisfies Jones' compatibility
conditions~\textrefp[J]{enu:j1}--\textrefp[J]{enu:j4} with $\kappa_J=(1-\eta_1)^{-1}$,
Remark~\ref{rem:projection} and Theorem~\ref{thm:projection} imply that the operators
$B^{(\varepsilon)} : SL^\infty_n\to SL^\infty_N$ and
$Q^{(\varepsilon)} : SL^\infty_N\to SL^\infty_n$ given by
\begin{equation}\label{eq:B+Q:operators}
  B^{(\varepsilon)} f
  = \sum_{i=1}^{2^{n+1}-1} \frac{\langle f, h_i\rangle}{\|h_i\|_2^2} b_i^{(\varepsilon)}
  \qquad\text{and}\qquad
  Q^{(\varepsilon)} g
  = \sum_{i=1}^{2^{n+1}-1}
  \frac{\langle g, b_i^{(\varepsilon)}\rangle}{\|b_i^{(\varepsilon)}\|_2^2} h_i
\end{equation}
satisfy the estimates
\begin{equation}\label{eq:B+Q:estimates}
  \|B^{(\varepsilon)} f \|_{SL^\infty} \leq \|f\|_{SL^\infty}
  \quad\text{and}\quad
  \|Q^{(\varepsilon)} g \|_{SL^\infty} \leq (1-\eta_1)^{-1/2}\|g\|_{SL^\infty},
\end{equation}
$f\in SL^\infty_n$, $g\in SL^\infty_N$.  By~\eqref{eq:B+Q:operators}, the operator
$P^{(\varepsilon)} : SL^\infty_N\to SL^\infty_N$ defined as
$f\mapsto B^{(\varepsilon)}Q^{(\varepsilon)} f$ has the form
\begin{equation}\label{eq:P:definition}
  P^{(\varepsilon)} f
  = \sum_{i=1}^{2^{n+1}-1}
  \frac{\langle f, b_i^{(\varepsilon)}\rangle}{\|b_i^{(\varepsilon)}\|_2^2} b_i^{(\varepsilon)},
  \qquad f\in SL^\infty_N.
\end{equation}
Therefore, $P^{(\varepsilon)}$ is an orthogonal projection, which, by~\eqref{eq:B+Q:estimates},
satisfies the estimate
\begin{equation}\label{eq:P:estimate}
  \|P^{(\varepsilon)} f\|_{SL^\infty}
  \leq (1-\eta_1)^{-1/2}\|f\|_{SL^\infty},
  \qquad f\in SL^\infty_N.
\end{equation}

Let $Y$ denote the subspace of $SL^\infty_N$ given by
\begin{equation*}
  Y = \Big\{g = \sum_{i=1}^{2^{n+1}-1} a_i b_i^{(\varepsilon)} :
  a_i\in \mathbb{R}\Big\},
\end{equation*}
and note that $Y$ is the image of the projection $P^{(\varepsilon)}$.
By~\eqref{rem:projection:identity:signs} and~\eqref{eq:B+Q:estimates},
$B^{(\varepsilon)} : SL^\infty_n\to Y$ is an isomorphism; thus we obtain the following commutative
diagram:
\begin{equation}\label{eq:commutative-diagram:preimage}
  \vcxymatrix{SL^\infty_n \ar[r]^{\Id_{SL^\infty_n}} \ar[d]_{B^{(\varepsilon)}} & SL^\infty_n\\
    Y \ar[r]_{\Id_Y} & Y \ar[u]_{{B^{(\varepsilon)}}^{-1}}},
\end{equation}
where $\|B^{(\varepsilon)}\| = 1$ and $\|{B^{(\varepsilon)}}^{-1}\| \leq (1-\eta_1)^{-1/2}$.  Now,
define $U : SL^\infty_N\to Y$ by
\begin{equation}\label{eq:almost-inverse}
  U f = \sum_{i=1}^{2^{n+1}-1}
  \frac{\langle f, b_i^{(\varepsilon)}\rangle}
  {\langle Tb_i^{(\varepsilon)}, b_i^{(\varepsilon)}\rangle}
  b_i^{(\varepsilon)},
\end{equation}
and note that by~\eqref{eq:proof:local-factor:ii:b}, the $1$-unconditionality of the Haar system in
$SL^\infty$ and~\eqref{eq:P:estimate}
\begin{equation}\label{eq:almost-inverse:bounded}
  \|U : SL^\infty_N\to Y\|
  \leq \frac{1}{\delta (1-\eta_1)^{1/2}}.
\end{equation}
Observe that for all $g = \sum_{i=1}^{2^{n+1}-1} a_i b_i^{(\varepsilon)} \in Y$, we have the
following identity:
\begin{equation}\label{eq:crucial-identity}
  UTg - g
  = \sum_{i=1}^{2^{n+1}-1}
  \sum_{j=1}^{i-1} a_j
  \frac{\langle T b_j^{(\varepsilon)}, b_i^{(\varepsilon)}\rangle}
  {\langle Tb_i^{(\varepsilon)}, b_i^{(\varepsilon)}\rangle}
  b_i^{(\varepsilon)}
  + \sum_{i=1}^{2^{n+1}-1} \sum_{j=1}^{i-1} a_i
  \frac{\big\langle b_i^{(\varepsilon)}, T^* b_j^{(\varepsilon)}\big\rangle}
  {\langle Tb_j^{(\varepsilon)}, b_j^{(\varepsilon)}\rangle}
  b_j^{(\varepsilon)}.
\end{equation}
Using the estimates~\eqref{eq:proof:local-factor:ii} and $|a_j| \leq \|g\|_{SL^\infty}$, together
with the above identity~\eqref{eq:crucial-identity} yields
\begin{equation}\label{eq:crucial-inequality}
  \|UTg - g\|_{SL^\infty}
  \leq \frac{\eta_1}{\delta} \Big(1+\frac{2^n}{1-\eta_1}\Big)\|g\|_{SL^\infty}.
\end{equation}

Finally, let $J : Y\to SL^\infty_N$ denote the operator given by $Jy = y$.  By our choice of
$\eta_1$ in~\eqref{eq:proof:local-factor:eta_1}, the operator $V : SL^\infty_N\to Y$ given by
$V=(UTJ)^{-1}U$ is well defined.  With these definitions, the following diagram is commutative:
\begin{equation}\label{eq:commutative-diagram:image}
  \vcxymatrix{
    Y \ar[rr]^{\Id_Y} \ar[dd]_J \ar[rd]_{UTJ} & & Y\\
    & Y \ar[ru]^{(UTJ)^{-1}} &\\
    SL^\infty_N \ar[rr]_T & & SL^\infty_N \ar[lu]_U \ar[uu]_V
  }
  \qquad \|J\|\|V\| \leq (1+\eta)/\delta.
\end{equation}
Merging the diagrams~\eqref{eq:commutative-diagram:preimage}
and~\eqref{eq:commutative-diagram:image} concludes the proof.\qedhere

\subsection{Preparation for the proof of Theorem~\ref{thm:local-primary}}\label{subsec:local-primary-prep}\hfill

\noindent
Before we come to the proof of Theorem~\ref{thm:local-primary}, we develop some notation and introduce
another combinatorial lemma (see~Lemma~\ref{lem:comb-2}).

Let $\mathcal{N}$ be a nested collection of sets with finite and positive measure.  Given
$\mathcal{X}\subset\mathcal{N}$, we define the \emph{founding generation}
$\mathcal{G}_0(\mathcal{X})$ of $\mathcal{X}$ by
\begin{equation}\label{eq:generation:0}
  \mathcal{G}_0(\mathcal{X})
  = \{ N\in \mathcal{X} : \text{there is no $M\in\mathcal{X}$ with $M\supsetneq N$}\}.
\end{equation}
We now inductively define the \emph{$k^{\text{th}}$ generation} $\mathcal{G}_k(\mathcal{X})$ of
$\mathcal{X}$.  Assuming that $\mathcal{G}_0(\mathcal{X}),\ldots,\mathcal{G}_{k-1}(\mathcal{X})$
have already been defined, we put
\begin{equation}\label{eq:generation:1}
  \mathcal{G}_k(\mathcal{X})
  = \mathcal{G}_0\Bigl(\mathcal{X}\setminus \bigl(
  \mathcal{G}_0(\mathcal{X})\cup\cdots\cup\mathcal{G}_{k-1}(\mathcal{X})\bigr)
  \Bigr),
  \qquad k\in\mathbb{N}.
\end{equation}
We denote the point-set of $\mathcal{G}_k(\mathcal{X})$ by $G_k(\mathcal{X})$, i.e.
\begin{equation}\label{eq:generation:2}
  G_k(\mathcal{X})
  = \bigcup \mathcal{G}_k(\mathcal{X}).
\end{equation}
Clearly, $G_k(\mathcal{X})\subset G_{k-1}(\mathcal{X})$, $k\in\mathbb{N}$.
The \emph{Carleson constant} $\cc{\mathcal{X}}$ of $\mathcal{X}$ is defined as
\begin{equation}\label{eq:carleson-constant}
  \cc{\mathcal{X}}
  = \sup_{N\in\mathcal{X}}\frac{1}{|N|}\sum_{\substack{M\in\mathcal{X}\\M\subset N}} |M|.
\end{equation}

The dyadic version of Lemma~\ref{lem:comb-2} (including a proof) can be found
in~\cite[Lemma~3.1.4]{mueller:2005}.
\begin{lem}\label{lem:comb-2}
  Let $\mathcal{N}$ be a collection of measurable, nested sets with finite and positive measure,
  $k\in\mathbb{N}$ and $\rho > 0$.  If $\mathcal{X}\subset\mathcal{N}$ is a collection with
  $\cc{\mathcal{X}} > \frac{k}{\rho}$, there exists a set $N_0\in\mathcal{N}$ such that
  \begin{equation*}
    \big|G_\ell\big(\{N\in\mathcal{X} : N\subset N_0\}\big)\big|
    > (1-\rho) |N_0|,
    \qquad 0\leq \ell \leq k.
  \end{equation*}
\end{lem}

\begin{proof}
  Although the proof is completely analogous the dyadic case, we give it here for sake of
  completeness.  Let $\mathcal{N}$ be a collection of nested sets with finite and positive measure,
  $k\in\mathbb{N}$ and $0 < \rho < 1$.  Suppose that the lemma fails, i.e.  there exists an
  $\ell\in\mathbb{N}_0$ with $0\leq \ell \leq k$ such that
  \begin{equation*}
    \big|G_\ell\big(\{N\in\mathcal{X} : N\subset N_0\}\big)\big|
    \leq (1-\rho) |N_0|,
    \qquad N_0\in\mathcal{N}.
  \end{equation*}
  Certainly, the above inequality implies that
  \begin{equation*}
    \big|G_k\big(\{N\in\mathcal{X} : N\subset N_0\}\big)\big|
    \leq (1-\rho) |N_0|,
    \qquad N_0\in\mathcal{N}.
  \end{equation*}

  Now, let $N_0\in\mathcal{N}$ be fixed, and put $\mathcal{X}_0=\{N\in\mathcal{X} : N\subset N_0\}$.
  By iterating the above inequality, we obtain
  \begin{equation*}
    \big|G_{mk+\ell}(\mathcal{X}_0)\big|
    \leq (1-\rho)^m |N_0|,
    \qquad m\in\mathbb{N}_0,\ 0\leq \ell \leq k-1.
  \end{equation*}
  Summing these estimates yields
  \begin{equation*}
    \sum_{N\in\mathcal{X}_0} |N|
    = \sum_{\ell=0}^{k-1}\sum_{m\in\mathbb{N}_0}\big|G_{mk+\ell}(\mathcal{X}_0)\big|
    \leq \sum_{\ell=0}^{k-1}\sum_{m\in\mathbb{N}_0} (1-\rho)^m |N_0|
    \leq \frac{k}{\rho} |N_0|.
  \end{equation*}
  Since $N_0\in\mathcal{X}$ was arbitrary, we obtain from the above inequality that
  $\cc{\mathcal{X}}\leq \frac{k}{\rho}$, which contradicts our hypothesis.
\end{proof}

The next Lemma is a monochromatic (and simplified) version of Jones' argument~\cite{jones:1985}.
\begin{lem}\label{lem:comb-3}
  Let $n\in\mathbb{N}_0$ and $\alpha,\beta\in\mathbb{R}$ be such that
  \begin{equation*}
    0 < \beta < 1
    \qquad\text{and}\qquad
    0 < \alpha < 2^{-n-1}\beta^{n+1}.
  \end{equation*}
Let $\mathcal{X}$ denote a collection of measurable nested
  sets with finite and positive measure which satisfies
  \begin{equation*}
    |G_n(\mathcal{X})| > (1 - \alpha)|G_0(\mathcal{X})|.
  \end{equation*}
  Then there exists a subcollection $\mathcal{Y}$ of $\mathcal{X}$ such that
  \begin{subequations}\label{eq:lem:comb-3}
    \begin{equation}\label{eq:lem:comb-3:a}
      |G_n(\mathcal{Y})| > \bigl( 1 - \alpha\frac{2^{n+1}}{\beta^{n+1}} \bigr) |G_0(\mathcal{X})|,
    \end{equation}
    and
    \begin{equation}\label{eq:lem:comb-3:b}
      |N\cap G_n(\mathcal{Y})|
      \geq (1 - \beta)|N|,
      \qquad N\in\mathcal{Y}.
    \end{equation}
  \end{subequations}
\end{lem}

\begin{proof}
  Put $\mathcal{F}_0 = \mathcal{G}_n(\mathcal{X})$, $F_0 = \bigcup \mathcal{F}_0$ and note that by
  our hypothesis we have the estimate
  \begin{equation}\label{eq:proof:lem:comb-3:1}
    |F_0| > (1-\alpha)|G_0(\mathcal{X})|.
  \end{equation}
  Let $j_0\in\mathbb{N}$ and assume that we already defined $\mathcal{F}_0,\ldots,\mathcal{F}_{j_0-1}$
  and $F_0,\ldots,F_{j_0-1}$.  To conclude the inductive step, we simply put
  \begin{equation}\label{eq:proof:lem:comb-3:2}
    \mathcal{F}_{j_0}
    = \{ N\in\mathcal{G}_{n-j_0}(\mathcal{X}) : |N\cap F_0\cap\dots\cap F_{j_0-1}|\geq (1-\beta)|N| \},
    \qquad F_{j_0} = \bigcup \mathcal{F}_{j_0}.
  \end{equation}
  Stopping the induction process after $n+1$ steps yields collections of pairwise disjoint sets
  $\mathcal{F}_j$, $0\leq j\leq n$, and $F_j = \bigcup \mathcal{F}_j$, $0\leq j\leq n$.

  Let $1\leq j\leq n$ be fixed.  We define
  \begin{equation}\label{eq:proof:lem:comb-3:3}
    \mathcal{E}_j
    = \mathcal{G}_{n-j}(\mathcal{X})\setminus \mathcal{F}_j,
    \qquad E_j = \bigcup \mathcal{E}_j,
  \end{equation}
  and note that by~\eqref{eq:proof:lem:comb-3:2} and~\eqref{eq:proof:lem:comb-3:3} we have
  $|N\cap F_0\cap\dots\cap F_{j-1}| < (1-\beta)|N|$, for all $N\in\mathcal{E}_j$.  Now, observe that
  \begin{equation*}
    |N|
    = |N\cap F_0\cap\dots\cap F_{j-1}| + |N\setminus (F_0\cap\dots\cap F_{j-1})|
    < (1-\beta)|N| + |N\setminus (F_0\cap\dots\cap F_{j-1})|,
  \end{equation*}
  for all $N\in\mathcal{E}_j$.  Summing the above estimate over all $N\in\mathcal{E}_j$ yields
  \begin{equation*}
    |E_j|
    \leq (1-\beta) |E_j| + |G_0(\mathcal{X})\setminus(F_0\cap\dots\cap F_{j-1})|.
  \end{equation*}
  Consequently, we obtain
  \begin{equation*}
    |E_j|
    \leq \frac{1}{\beta} |G_0(\mathcal{X})\setminus(F_0\cap\dots\cap F_{j-1})|.
  \end{equation*}
  Combining $(1-\alpha)|G_0(\mathcal{X})| \leq |G_{n-j}(\mathcal{X})| = |E_j| + |F_j|$ with the
  latter estimate gives us
  \begin{equation}\label{eq:proof:lem:comb-3:4}
    |F_j|
    \geq (1-\alpha)|G_0(\mathcal{X})|
    - \frac{1}{\beta} |G_0(\mathcal{X})\setminus(F_0\cap\dots\cap F_{j-1})|,
    \qquad 1\leq j\leq n.
  \end{equation}

  We claim that the following inequality holds:
  \begin{equation}\label{eq:proof:lem:comb-3:5}
    |F_j|
    \geq (1-\alpha\frac{2^j}{\beta^j})|G_0(\mathcal{X})|,
    \qquad 0\leq j\leq n.
  \end{equation}
  We will prove the inequality~\eqref{eq:proof:lem:comb-3:5} by induction on $j$.
  By~\eqref{eq:proof:lem:comb-3:1}, the inequality~\eqref{eq:proof:lem:comb-3:5} is true for $j=0$.
  Now let $0\leq j_0\leq n-1$ and assume we have already proved that~\eqref{eq:proof:lem:comb-3:5}
  holds for all $0\leq j\leq j_0$.  By~\eqref{eq:proof:lem:comb-3:4} and our induction
  hypothesis we obtain
  \begin{align*}
    |F_{j_0+1}|
    & \geq (1-\alpha)|G_0(\mathcal{X})|
      - \frac{1}{\beta} |G_0(\mathcal{X})\setminus(F_0\cap\dots\cap F_{j_0})|\\
    & \geq (1-\alpha)|G_0(\mathcal{X})|
      - \frac{1}{\beta} \sum_{j=0}^{j_0} |G_0(\mathcal{X})\setminus F_j|\\
    & \geq (1-\alpha)|G_0(\mathcal{X})|
      - \frac{1}{\beta} \sum_{j=0}^{j_0} \alpha\frac{2^j}{\beta^j} |G_0(\mathcal{X})|.
  \end{align*}
  Using $\alpha\leq \frac{\alpha}{\beta^{j_0+1}}$ to estimate the first term and
  $\frac{1}{\beta} \sum_{j=0}^{j_0} \alpha\frac{2^j}{\beta^j}\leq
  \alpha\frac{2^{j_0+1}-1}{\beta^{j_0+1}}$ for the second term yields~\eqref{eq:proof:lem:comb-3:5}
  for $j_0+1$.  Thus, we proved~\eqref{eq:proof:lem:comb-3:5}.

  By~\eqref{eq:proof:lem:comb-3:5}, we obtain
  \begin{equation}\label{eq:proof:lem:comb-3:6}
    |F_0\cap\dots\cap F_n|
    \geq \bigl( 1 - \alpha \sum_{j=0}^n \frac{2^j}{\beta^j}\bigr) |G_0(\mathcal{X})|
    \geq \bigl( 1 - \alpha\frac{2^{n+1}}{\beta^{n+1}} \bigr) |G_0(\mathcal{X})|.
  \end{equation}
  Finally, we define the subcollection $\mathcal{Y}$ of $\mathcal{X}$ by putting
  \begin{equation*}
    \mathcal{Y}
    = \{ N : N\in\mathcal{F}_j,\ N\cap F_0\cap\dots\cap F_n \neq \emptyset,\ 0\leq j\leq n\}.
  \end{equation*}
  We will now verify that $\mathcal{Y}$ satisfies~\eqref{eq:lem:comb-3}.  First, we will show that
  \begin{equation}\label{eq:proof:lem:comb-3:7}
    G_n(\mathcal{Y}) = F_0\cap\dots\cap F_n.
  \end{equation}
  To this end, let $x\in G_n(\mathcal{Y})$ and note that in this case, there exist sets
  $N_j\in\mathcal{Y}$, $0\leq j\leq n$, with $x\in N_0\subsetneq \dots \subsetneq N_n$.  Clearly,
  this is only possible if $N_j\in\mathcal{F}_j$, $0\leq j\leq n$.  Now, since $N_0\in\mathcal{Y}$,
  we have $N_0\cap F_j\neq \emptyset$, $0\leq j\leq n$, which implies $N_0\subset F_j$,
  $0\leq j\leq n$.  Consequently, $x\in N_0\subset F_0\cap \dots\cap F_n$.  If on the other hand
  $x\in F_0\cap \dots\cap F_n$, then we know that there exist $N_j\in\mathcal{F}_j$ with $N_j\ni x$,
  $0\leq j\leq n$.  Hence, $x\in N_0\subsetneq \dots\subsetneq N_n$ and
  $N_j\cap F_0\cap\dots\cap F_n\neq \emptyset$, $0\leq j\leq n$.  Therefore, we obtain
  $N_j\in\mathcal{Y}$ for all $0\leq j\leq n$, and that $N_0\in\mathcal{G}_n(\mathcal{Y})$, which
  shows that $x\in N_0\subset G_n(\mathcal{Y})$.

  Second, note that combining~\eqref{eq:proof:lem:comb-3:7} with~\eqref{eq:proof:lem:comb-3:6}
  yields~\eqref{eq:lem:comb-3:a}.

  Thirdly, let $N\in\mathcal{Y}$.  Thus, there exists an integer $j_0$ with $0\leq j_0\leq n$ such
  that $N\in\mathcal{F}_{j_0}$, and $N\cap F_0\cap\dots\cap F_n\neq\emptyset$.  Now, observe that
  whenever $j\geq j_0$ and $N\cap F_j\neq \emptyset$, then $N\subset F_j$.  Hence,
  $N\cap F_0\cap\dots\cap F_n = N\cap F_0\cap\dots\cap F_{j_0-1}$, and we
  obtain~\eqref{eq:lem:comb-3:b} by~\eqref{eq:proof:lem:comb-3:2} and~\eqref{eq:proof:lem:comb-3:7}.
\end{proof}

\subsection{Proof of Theorem~\ref{thm:local-primary}}\label{subsec:local-primary}\hfill

\noindent
The framework for the proof of Theorem~\ref{thm:local-primary} is similar to that
of~\cite[Theorem~2.2]{lechner:2016:factor-SL}, although here we certainly use quantitative, finite
dimensional techniques instead of qualitative, infinite dimensional techniques.

\begin{myproof}[Proof of Theorem~\ref{thm:local-primary}]
  Let $n\in\mathbb{N}_0$ and $\Gamma,\eta > 0$.  Let $\eta_1 = \eta_1(n,\eta)$ denote the largest
  positive constant satisfying the inequalities
  \begin{equation}\label{eq:proof:local-primary:eta_1}
    0 < \eta_1 \leq 1/2,
    \qquad
    \eta_14^{n+3}n \leq 1/2,
    \qquad
    \frac{1}{1 - \eta_14^{n+4}n} \leq 1 + \eta.
  \end{equation}
  Define the integers
  \begin{equation}\label{eq:proof:local-primary:dim}
    n_1=\Bigl\lfloor \Bigl(\frac{32n}{\eta_1}\Bigr)^{n+2}\Bigr\rfloor + 1,
    \qquad
    N=N(n_1,\Gamma,\eta_1).
  \end{equation}
  where $N(n_1,\Gamma,\eta_1)$ is the integer in Theorem~\ref{thm:quasi-diag} with the parameters $n_1$,
  $\Gamma$, $\eta_1$ and $\delta=0$.  Note that by the above definitions of $n_1$ and $\eta_1$, we
  have actually that $N$ depends only on $n$, $\Gamma$ and $\eta$, i.e.  $N=N(n,\Gamma,\eta)$.
  Finally, let $T : SL^\infty_N\to SL^\infty_N$ with $\|T\|\leq \Gamma$ be fixed throughout the rest
  of the proof.

  The proof will be divided into the following four steps.
  \begin{enumerate}[Step 1]
  \item By Theorem~\ref{thm:quasi-diag}, there exists a finite sequence of collections
    $(\mathcal{B}_K : K\in\mathcal{D}^{n_1})$ with $\mathcal{B}_K\subset \mathcal{D}^N$,
    $K\in\mathcal{D}^{n_1}$, satisfying Jones' compatibility
    conditions~\textrefp[J]{enu:j1}--\textrefp[J]{enu:j4}.  These collections generate the block
    basis $(b_K : K\in\mathcal{D}^{n_1})$, which simultaneously almost-diagonalizes the operators
    $T$ and $\Id_{SL^\infty_N}-T$.
    
  \item One of the two collections
    \begin{align*}
      \mathcal{M}
      & = \Big\{B_K : K\in\mathcal{D}^{n_1},\
        \langle T b_K, b_K\rangle\geq \frac{\|b_K\|_2^2}{2}\Big\},\\
      \mathcal{N}
      & = \Big\{B_K : K \in\mathcal{D}^{n_1},\
        \langle(\Id_{SL^\infty_N} - T) b_K, b_K\rangle
        \geq\frac{\|b_K\|_2^2}{2}\Big\},
    \end{align*}
    contains a finite sequence of subcollections $(\mathcal{C}_I : I\in \mathcal{D}^n)$ that
    satisfies Jones' compatibility conditions~\textrefp[J]{enu:j1}--\textrefp[J]{enu:j4}.
    
  \item Consequently, by the reiteration Theorem~\ref{thm:projection-iteration} and Theorem~\ref{thm:projection},
    the block basis $(\widetilde b_I : I\in\mathcal{D}^n)$ given by
    \begin{equation*}
      \widetilde b_I
      = \sum_{B_K\in\mathcal{C}_I} b_K,
      \qquad I\in\mathcal{D}^n,
    \end{equation*}
    spans a complemented copy of $SL^\infty_n$ (the constants for the norms of the isomorphism and
    the projection do not depend on $n$).  Moreover, the operators $T$ and $\Id_{SL^\infty_N}-T$ are
    both almost-diagonalized by $(\widetilde b_I : I\in\mathcal{D}^n)$, and either $T$ or
    $\Id_{SL^\infty_N}-T$ has large diagonal with respect to $(\widetilde b_I : I\in\mathcal{D}^n)$
    (depending on whether we selected $\mathcal{C}_I\subset \mathcal{M}$, $I\in\mathcal{D}^n$ or
    $\mathcal{C}_I\subset \mathcal{N}$, $I\in\mathcal{D}^n$, in the previous step).
  
  \item Finally, repeating the proof of Theorem~\ref{thm:local-factor} with $\widetilde b_I$ instead of
    $b_I$ yields the commutative diagram
    \begin{equation*}
      \vcxymatrix{SL^\infty_n \ar[r]^{\Id_{SL^\infty_n}} \ar[d]_R & SL^\infty_n\\
        SL^\infty_N \ar[r]_H & SL^\infty_N \ar[u]_S}
      \qquad \|R\|\|S\|\leq 2 + \eta,
    \end{equation*}
    where $H$ is either $T$ or $\Id_{SL^\infty_N} - T$.
  \end{enumerate}

  \begin{proofstep}\label{step:local-primary:1}\hfill\\
    By the definition of $N$ (see~\eqref{eq:proof:local-primary:dim}), the almost-diagonalization
    Theorem~\ref{thm:quasi-diag} (and Remark~\ref{rem:quasi-diag}~\eqref{enu:rem:quasi-diag:i}
    and~\eqref{enu:rem:quasi-diag:ii}) with parameters $n_1$, $\Gamma$, $\eta_1$ and $\delta=0$, we
    obtain a finite block basic sequence $(b_K : K\in\mathcal{D}^{n_1})$ given by
    \begin{equation}\label{eq:proof:thm:local-primary:block-basis:1}
      b_K
      = \sum_{Q\in\mathcal{B}_K} h_Q,
      \qquad K\in\mathcal{D}^{n_1},
    \end{equation}
    which has the following properties:
    \begin{enumerate}[(i)]
    \item\label{enu:proof:thm:local-primary:block-basis:i} $\mathcal{B}_K\subset \mathcal{D}^N$,
      $K\in\mathcal{D}^{n_1}$, and the finite sequence of collections
      $(\mathcal{B}_K : K\in \mathcal{D}^{n_1})$ satisfies Jones' compatibility
      conditions~\textrefp[J]{enu:j1}--\textrefp[J]{enu:j4} with constant
      $\kappa_J = (1-\eta_1)^{-1}$, and
      \begin{equation}\label{eq:proof:thm:local-primary:block-basis:i}
        (1-\eta_1) |K|
        \leq |B_K|
        \leq |K|,
        \qquad K\in\mathcal{D}^{n_1}.
      \end{equation}
    
    \item\label{enu:proof:thm:local-primary:block-basis:ii} For any $1\leq i \leq 2^{n_1+1}-1$ we
      have the estimate
      \begin{subequations}\label{eq:proof:thm:local-primary-ii}
        \begin{equation}\label{eq:proof:thm:local-primary-ii:a}
          \sum_{j=1}^{i-1} |\langle T b_j, b_i\rangle|
          + |\langle b_i, T^* b_j\rangle|
          \leq \eta_1\|b_i\|_2^2.
        \end{equation}
        Since $\langle b_j, b_i\rangle = 0$, $i\neq j$, we also have
        \begin{equation}\label{eq:proof:thm:local-primary-ii:b}
          \sum_{j=1}^{i-1} |\langle (\Id_{SL^\infty_N} - T) b_j, b_i\rangle|
          + |\langle b_i, (\Id_{SL^\infty_N} - T)^* b_j\rangle|
          \leq \eta_1\|b_i\|_2^2.
        \end{equation}
      \end{subequations}
    \end{enumerate}
  \end{proofstep}

  \begin{proofstep}\label{step:local-primary:2}\hfill\\
    We define the collections of measurable, nested sets with finite and positive measure
    \begin{align*}
      \mathcal{M}
      & = \Big\{B_K : K\in\mathcal{D}^{n_1},\
        \langle T b_K, b_K\rangle\geq \frac{\|b_K\|_2^2}{2}\Big\},\\
      \mathcal{N}
      & = \Big\{B_K : K \in\mathcal{D}^{n_1},\
        \langle(\Id_{SL^\infty_N} - T) b_K, b_K\rangle\geq
        \frac{\|b_K\|_2^2}{2}\Big\}.
    \end{align*}
    Note that by~\eqref{eq:proof:thm:local-primary:block-basis:i} we have
    $(1-\eta_1)n_1 \leq \cc{\mathcal{M}\cup \mathcal{N}}\leq \cc{\mathcal{M}} + \cc{\mathcal{N}}$
    (see~\eqref{eq:carleson-constant} for a definition of $\cc{\cdot}$).  If
    $\cc{\mathcal{M}}\geq (1-\eta_1)\frac{n_1}{2}$, we define $\mathcal{L}=\mathcal{M}$ and $H=T$;
    otherwise, we put $\mathcal{L}=\mathcal{N}$ and $H=\Id_{SL^\infty_N}-T$.
    We note the estimate
    \begin{equation}\label{eq:proof:thm:local-primary:large-coll-diag}
      \cc{\mathcal{L}}\geq (1-\eta_1)\frac{n_1}{2}
      \qquad\text{and}\qquad
      \langle H b_I, b_I\rangle\geq \frac{\|b_I\|_2^2}{2},
      \ I\in\mathcal{L}.
    \end{equation}
    By~\eqref{eq:proof:thm:local-primary:large-coll-diag} and~\eqref{eq:proof:local-primary:dim} we
    have that $\cc{\mathcal{L}}\geq(1-\eta_1)\frac{n_1}{2} > \frac{n}{(\eta_1/(32n))^{n+1}}$, thus
    Lemma~\ref{lem:comb-2} (with a parameter setting of $k=n$ and $\rho=(\frac{\eta_1}{32n})^{n+1}$)
    implies that there exists a set $B_0\in\mathcal{L}$ such that
    \begin{equation}\label{eq:proof:thm:local-primary:ind:0}
      \big|G_k\big(\{B\in\mathcal{\mathcal{L}} : B\subset B_0\}\big)\big|
      > \bigl(1-\bigl(\frac{\eta_1}{32n}\bigr)^{n+1}\bigr) |B_0|,
      \qquad 0\leq k \leq n.
    \end{equation}
    Put $\mathcal{L}_0 = \{B\in\mathcal{L} : B\subset B_0\}$. Since
    $\bigl(\frac{\eta_1}{32n}\bigr)^{n+1} < 2^{-n-1}\bigl(\frac{\eta_1}{8n}\bigr)^{n+1}$,
    Lemma~\ref{lem:comb-3} yields a subcollection $\mathcal{L}_1$ of $\mathcal{L}_0$ containing $B_0$
    such that
    \begin{equation}\label{eq:proof:thm:local-primary:density:0}
      | B\cap G_k(\mathcal{L}_1) |
      \geq \bigl(1 - \frac{\eta_1}{8n}\bigr)|B|,
      \qquad B\in\mathcal{L}_1,\ 0\leq k\leq n.
    \end{equation}

    We will now inductively define a finite sequence of collections
    $(\mathcal{C}_I : I\in\mathcal{D}^n)$ with $\mathcal{C}_I\subset\mathcal{L}_1$,
    $I\in\mathcal{D}^n$.  To begin, we simply put
    \begin{equation}\label{eq:proof:thm:local-primary:ind:1-top}
      \mathcal{C}_1
      = \mathcal{C}_{[0,1)}
      = \{B_0\}.
    \end{equation}
    Let us assume that we have already defined the collections
    $\mathcal{C}_1,\ldots,\mathcal{C}_{i_0-1}$.  We will now construct $C_{i_0}$.  To this end, let
    $I_0\in\mathcal{D}^n$ with $\mathcal{O}(I_0) = i_0$, and define
    \begin{equation*}
      C_{\widetilde I_0}^\ell
      = \bigcup \bigl\{ B_K^\ell : B_K\in \mathcal{C}_{\widetilde I_0} \bigr\}
      \qquad\text{and}\qquad
      C_{\widetilde I_0}^r
      = \bigcup \bigl\{ B_K^r : B_K\in \mathcal{C}_{\widetilde I_0} \bigr\},
    \end{equation*}
    where the sets $B_K^\ell$ and $B_K^r$ are given by
    \begin{equation*}
      B_K^\ell = \bigcup \bigl\{Q^\ell : Q \in \mathcal{B}_K\bigr\}
      \qquad\text{and}\qquad
      B_K^r = \bigcup \bigl\{Q^r : Q \in \mathcal{B}_K\bigr\}.
    \end{equation*}
    If $I_0$ is the \emph{left half} of $\widetilde I_0$ and $|I_0|=2^{-k_0}$, $k_0\leq n$, we put
    \begin{subequations}\label{eq:proof:thm:local-primary:ind:1}
      \begin{equation}\label{eq:proof:thm:local-primary:ind:1:a}
        \mathcal{C}_{I_0}
        = \{ B\in\mathcal{G}_{k_0}(\mathcal{L}_1) : B\subset C_{\widetilde I_0}^\ell \},
      \end{equation}
      and if $I_0$ is the \emph{right half} of $\widetilde I_0$ and $|I_0|=2^{-k_0}$, we define
      \begin{equation}\label{eq:proof:thm:local-primary:ind:1:b}
        \mathcal{C}_{I_0}
        = \{ B\in\mathcal{G}_{k_0}(\mathcal{L}_1) : B\subset C_{\widetilde I_0}^r \}.
      \end{equation}
    \end{subequations}
    We stop the induction after the construction of the collections $\mathcal{C}_I$,
    $I\in\mathcal{D}^n$.

    It is easily verified that $(\mathcal{C}_I : I\in\mathcal{D}^n)$
    satisfies~\textrefp[J]{enu:j1}--\textrefp[J]{enu:j3}.  We will now verify that
    $(\mathcal{C}_I : I\in\mathcal{D}^n)$ satisfies~\textrefp[J]{enu:j4} with constant
    $(1-\eta_1)^{-1}$.  First, note that by the inductive construction above
    (see~\eqref{eq:proof:thm:local-primary:ind:1-top} and~\eqref{eq:proof:thm:local-primary:ind:1}),
    we have $\mathcal{G}_n(\mathcal{L}_1) = \bigcup_{I\in\mathcal{D}_n}\mathcal{C}_I$, and
    thus,~\eqref{eq:proof:thm:local-primary:density:0} gives us
    \begin{equation}\label{eq:proof:thm:local-primary:density:1}
      \bigl| B\cap \bigcup_{I\in\mathcal{D}_n} C_I \bigr|
      \geq \bigl(1 - \frac{\eta_1}{8n}\bigr) |B|,
      \qquad B\in\bigcup_{I\in\mathcal{D}^n}\mathcal{C}_I.
    \end{equation}
    Second, let $I\in\mathcal{D}^n\setminus \{[0,1)\}$ and $B\in\mathcal{C}_{\widetilde I}$ be
    fixed (if $I=[0,1)$, there is nothing to show).  Choose $J\in\mathcal{D}^n\setminus\{[0,1)\}$
    such that $I\cup J = \widetilde I = \widetilde J$.  Hence,
    by~\eqref{eq:proof:thm:local-primary:density:1} and~\textrefp[J]{enu:j3} we obtain
    \begin{equation*}
      \bigl(1 - \frac{\eta_1}{8n}\bigr) |B|
      \leq | B\cap C_I | + | B\cap C_J |.
    \end{equation*}
    Considering~\eqref{eq:proof:thm:local-primary:ind:1} and that $|B^\ell|\leq |B|/2$,
    $|B^r|\leq |B|/2$, the above estimate yields
    \begin{equation*}
      \bigl(1 - \frac{\eta_1}{8n}\bigr) |B|
      \leq
      \begin{cases}
        | B^\ell\cap C_I | + |B|/2, & \text{if $I$ is the left half of $\widetilde I$},\\
        | B^r\cap C_I | + |B|/2, & \text{if $I$ is the right half of $\widetilde I$}.
      \end{cases}
    \end{equation*}
    Since either $B\cap C_I = B^\ell\cap C_I$ or $B\cap C_I = B^r\cap C_I$, the latter estimate
    gives us
    \begin{equation}\label{eq:proof:thm:local-primary:coll-prop:1}
      \frac{1}{2}\bigl(1-\frac{\eta_1}{4n}\bigr) |B|
      \leq |B\cap C_I|
      \leq \frac{1}{2} |B|,
      \qquad I\in\mathcal{D}^n\setminus\{[0,1)\},\ B\in\mathcal{C}_{\widetilde I}.
    \end{equation}
    The estimate on the right hand side is obvious from the principle of our construction
    (see~\eqref{eq:proof:thm:local-primary:ind:1-top} and~\eqref{eq:proof:thm:local-primary:ind:1}).
    Summing~\eqref{eq:proof:thm:local-primary:coll-prop:1} over all $B\in\mathcal{C}_{\widetilde I}$
    yields together with~\textrefp[J]{enu:j2} and~\textrefp[J]{enu:j3} that
    \begin{equation*}
      \frac{1}{2}\bigl(1-\frac{\eta_1}{4n}\bigr) |C_{\widetilde I}|
      \leq |C_I|
      \leq \frac{1}{2} |C_{\widetilde I}|,
      \qquad I\in\mathcal{D}^n, I\neq [0,1).
    \end{equation*}
    By iterating the latter inequality we obtain
    \begin{equation}\label{eq:proof:thm:local-primary:coll-prop:2}
      (1-\frac{\eta_1}{4}) |I|
      \leq \frac{|C_I|}{ |C_{[0,1)}|}
      \leq |I|,
      \qquad I\in\mathcal{D}^n.
    \end{equation}
    Let $I_0,I\in\mathcal{D}^n$ with $I_0\subset I$, $4|I_0|\leq |I|$, and $B\in\mathcal{C}_I$,
    then~\textrefp[J]{enu:j2} and~\textrefp[J]{enu:j3} imply
    \begin{equation}\label{eq:proof:thm:local-primary:coll-prop:2:iter:1}
      |B\cap C_{I_0}|
      = |B\cap C_{\widetilde I_0}\cap C_{I_0}|
      = \sum_{A\in \mathcal{C}_{\widetilde{I_0}}}\big|B\cap A\cap C_{I_0}\big|.
    \end{equation}
    Since $\widetilde I_0\subset I$, we know that $A\subset B$ whenever
    $A\in\mathcal{C}_{\widetilde I_0}$, $B\in\mathcal{C}_I$ and $B\cap A\neq \emptyset$.
    Hence,~\eqref{eq:proof:thm:local-primary:coll-prop:2:iter:1},
    \eqref{eq:proof:thm:local-primary:coll-prop:1} and~\textrefp[J]{enu:j2} yield
    \begin{equation}\label{eq:proof:thm:local-primary:coll-prop:2:iter:2}
      |B\cap C_{I_0}|
      = \sum_{\substack{A\in \mathcal{C}_{\widetilde{I_0}}\\A\subset B}}\big|A\cap C_{I_0}\big|
      \geq \frac{1}{2}\bigl(1-\frac{\eta_1}{4n}\bigr) |B\cap C_{\widetilde I_0}|.
    \end{equation}
    Iterating~\eqref{eq:proof:thm:local-primary:coll-prop:2:iter:1}
    and~\eqref{eq:proof:thm:local-primary:coll-prop:2:iter:2} while
    using~\eqref{eq:proof:thm:local-primary:coll-prop:1} in each of those iterations, we obtain
    \begin{equation*}
      |B\cap C_{I_0}|
      \geq \bigl(1-\frac{\eta_1}{4}\bigr) \frac{|I_0|}{|I|} |B\cap C_{[0,1)}|
      \qquad B\in\mathcal{C}_I.
    \end{equation*}
    Combining the latter estimate with~\eqref{eq:proof:thm:local-primary:coll-prop:2} and noting
    that $B\subset C_{[0,1)}$ yields
    \begin{equation}\label{eq:proof:thm:local-primary:coll-prop:3}
      |B\cap C_{I_0}|
      \geq \bigl(1-\frac{\eta_1}{2}\bigr) \frac{|C_{I_0}|}{|C_I|} |B|,
      \qquad B\in\mathcal{C}_I,
    \end{equation}
    whenever $I_0,I\in\mathcal{D}^n$ with $I_0\subset I$.
    By~\eqref{eq:proof:thm:local-primary:coll-prop:3}, the finite sequence
    $(\mathcal{C}_I : I\in\mathcal{D}^n)$ satisfies Jones' compatibility
    conditions~\textrefp[J]{enu:j1}--\textrefp[J]{enu:j4} with constant
    $\kappa_J=(1-\frac{\eta_1}{2})^{-1}$.  Furthermore, in \textref[Step~]{step:local-primary:1} we
    showed that $(\mathcal{B}_K : K\in \mathcal{D}^{n_1})$ satisfies Jones' compatibility
    conditions~\textrefp[J]{enu:j1}--\textrefp[J]{enu:j4} with constant $\kappa_J=(1-\eta_1)^{-1}$.
    Consequently, if we put
    \begin{equation}\label{eq:proof:thm:local-primary:tilde-coll:1}
      \mathcal{\widetilde B}_I
      = \bigcup_{B_K\in\mathcal{C}_I} \mathcal{B}_K
      \quad\text{and}\quad
      \widetilde B_I
      = \bigcup \mathcal{\widetilde B}_I,
      \qquad I\in\mathcal{D}^n,
    \end{equation}
    the reiteration Theorem~\ref{thm:projection-iteration} implies that
    $(\mathcal{\widetilde B}_I : I\in\mathcal{D}^n)$,
    satisfies~\textrefp[J]{enu:j1}--\textrefp[J]{enu:j4} with constant $\kappa_J=(1-\eta_1)^{-2}$.
  \end{proofstep}

  \begin{proofstep}\hfill\\
    Define the block basis $(\widetilde b_I : I\in\mathcal{D}^n)$ of $(b_K:K\in\mathcal{D}^{n_1})$
    by putting
    \begin{equation}\label{eq:proof:thm:local-primary:bb:1}
      \widetilde b_I
      = \sum_{B_K\in\mathcal{C}_I} b_K
      = \sum_{B_K\in\mathcal{C}_I} \sum_{Q\in\mathcal{B}_K} h_Q
      = \sum_{Q\in\mathcal{\widetilde B}_I} h_Q,
      \qquad I\in\mathcal{D}^n.
    \end{equation}
    Now, let $I\in\mathcal{D}^n$ be fixed.  By~\textrefp[J]{enu:j2},
    \eqref{eq:proof:thm:local-primary:large-coll-diag} and~\eqref{eq:proof:thm:local-primary-ii}, we
    obtain the diagonal estimate
    \begin{equation*}
      \begin{aligned}
        \langle H \widetilde b_I, \widetilde b_I\rangle & \geq \frac{1}{2}\|\widetilde b_I\|_2^2 -
        \sum_{\substack{B_K,B_{L}\in\mathcal{C}_I\\\mathcal{O}(K)<\mathcal{O}(L)}}
        |\langle H b_K, b_{L}\rangle| + |\langle b_{L}, H^*b_K\rangle|\\
        & \geq \frac{1}{2}\|\widetilde b_I\|_2^2 - \eta_1 \sum_{B_{L}\in\mathcal{C}_I} \|b_{L}\|_2^2
        = \big(\frac{1}{2}-\eta_1\big)\|\widetilde b_I\|_2^2.
      \end{aligned}
    \end{equation*}
    We summarize what we proved so far:
    \begin{equation}\label{eq:proof:thm:local-primary:bb:diag}
      \langle H \widetilde b_I, \widetilde b_I\rangle
      \geq \big(\frac{1}{2}-\eta_1\big)\|\widetilde b_I\|_2^2,
      \qquad I\in\mathcal{D}^n.
    \end{equation}

    For the off diagonal estimate, we define
    \begin{equation*}
      \mathcal{C}
      = \bigcup\big\{\mathcal{C}_J : J\in\mathcal{D}^n,\ \mathcal{O}(J) < \mathcal{O}(I)\big\}.
    \end{equation*}
    and note that by~\textrefp[J]{enu:j2} we have the estimate
    \begin{equation}\label{eq:proof:thm:local-primary:off-diag:1}
      \sum_{\substack{J\in\mathcal{D}^n\\\mathcal{O}(J)<\mathcal{O}(I)}}
      |\langle H \widetilde b_J, \widetilde b_I\rangle|
      \leq \sum_{B_L\in\mathcal{C}_I}
      \sum_{\substack{B_K\in\mathcal{C}\\\mathcal{O}(K) < \mathcal{O}(L)}}
      |\langle H b_K, b_L\rangle|
      +  \sum_{B_L\in\mathcal{C}_I}
      \sum_{\substack{B_K\in\mathcal{C}\\\mathcal{O}(K) > \mathcal{O}(L)}}
      |\langle b_K, H^*b_L\rangle|.
    \end{equation}
    Estimating the first sum on the right hand side
    of~\eqref{eq:proof:thm:local-primary:off-diag:1} by~\eqref{eq:proof:thm:local-primary-ii} yields
    \begin{equation}\label{eq:proof:thm:local-primary:off-diag:2}
      \sum_{B_L\in\mathcal{C}_I}
      \sum_{\substack{B_K\in\mathcal{C}\\\mathcal{O}(K) < \mathcal{O}(L)}}
      |\langle H b_K, b_L\rangle|
      \leq \eta_1 \sum_{B_L\in\mathcal{C}_I} \|b_L\|_2^2
      = \eta_1 \|\widetilde b_I\|_2^2.
    \end{equation}
    The latter equality follows from~\textrefp[J]{enu:j2}.
    By~\eqref{eq:proof:thm:local-primary-ii}, we obtain the following estimate for the second term
    on the right hand side of~\eqref{eq:proof:thm:local-primary:off-diag:1}:
    \begin{equation}\label{eq:proof:thm:local-primary:off-diag:3}
      \sum_{B_L\in\mathcal{C}_I}
      \sum_{\substack{B_K\in\mathcal{C}\\\mathcal{O}(K) > \mathcal{O}(L)}}
      |\langle b_K, H^*b_L\rangle|
      =  \sum_{B_K\in\mathcal{C}}
      \sum_{\substack{B_L\in\mathcal{C}_I\\\mathcal{O}(L) < \mathcal{O}(K)}}
      |\langle b_K, H^*b_L\rangle|
      \leq \eta_1 \sum_{B_K\in\mathcal{C}} \|b_K\|_2^2.
    \end{equation}
    \textrefp[J]{enu:j2} and~\eqref{eq:proof:thm:local-primary:block-basis:i} gives us
    \begin{equation}\label{eq:proof:thm:local-primary:off-diag:4}
      \sum_{B_K\in\mathcal{C}} \|b_K\|_2^2
      \leq \sum_{B_K\in\mathcal{C}} |K|
      \leq \sum_{J\in\mathcal{D}^n} |C_J|
      \leq n |C_{[0,1)}|.
    \end{equation}
    Combining~\eqref{eq:proof:thm:local-primary:off-diag:3}
    with~\eqref{eq:proof:thm:local-primary:off-diag:4} yields
    \begin{equation}\label{eq:proof:thm:local-primary:off-diag:5}
      \sum_{B_L\in\mathcal{C}_I}
      \sum_{\substack{B_K\in\mathcal{C}\\\mathcal{O}(K) > \mathcal{O}(L)}}
      |\langle b_K, H^*b_L\rangle|
      \leq \eta_1 n |C_{[0,1)}|.
    \end{equation}
    Collecting the estimates~\eqref{eq:proof:thm:local-primary:off-diag:5}
    \eqref{eq:proof:thm:local-primary:off-diag:2} and~\eqref{eq:proof:thm:local-primary:off-diag:1}
    we obtain
    \begin{equation}\label{eq:proof:thm:local-primary:off-diag:6}
      \sum_{\substack{J\in\mathcal{D}^n\\\mathcal{O}(J)<\mathcal{O}(I)}}
      |\langle H \widetilde b_J, \widetilde b_I\rangle|
      \leq 2 \eta_1 n |C_{[0,1)}|
    \end{equation}
    Repeating the above argument with the roles of $H$ and $H^*$ reversed yields
    \begin{equation}\label{eq:proof:thm:local-primary:off-diag:7}
      \sum_{\substack{J\in\mathcal{D}^n\\\mathcal{O}(J)<\mathcal{O}(I)}}
      |\langle \widetilde b_I, H^* \widetilde b_J\rangle|
      \leq 2 \eta_1 n |C_{[0,1)}|
    \end{equation}
    Adding~\eqref{eq:proof:thm:local-primary:off-diag:6}
    and~\eqref{eq:proof:thm:local-primary:off-diag:7} yields our desired off-diagonal estimate
    \begin{equation}\label{eq:proof:thm:local-primary:off-diag:8}
      \sum_{\substack{J\in\mathcal{D}^n\\\mathcal{O}(J)<\mathcal{O}(I)}}
      |\langle H \widetilde b_J, \widetilde b_I\rangle|
      + |\langle \widetilde b_I, H^* \widetilde b_J\rangle|
      \leq 4 \eta_1 n |C_{[0,1)}|,
      \qquad I\in\mathcal{D}^n.
    \end{equation}
  \end{proofstep}

  \begin{proofstep}\hfill\\
    As usual, we identify $\widetilde b_i = \widetilde b_I$, whenever $\mathcal{O}(I)=i$.  Thus,
    \eqref{eq:proof:thm:local-primary:off-diag:8} and~\eqref{eq:proof:thm:local-primary:bb:diag}
    read as follows:
    \begin{subequations}\label{eq:proof:thm:local-primary:all-estimates}
      \begin{align}
        \sum_{j=1}^{i-1}
        |\langle H \widetilde b_j, \widetilde b_i\rangle|
        + |\langle \widetilde b_i, H^* \widetilde b_j\rangle|
        & \leq 4 \eta_1 n |C_{[0,1)}|,
        && 1\leq i \leq 2^{n+1}-1,
           \label{eq:proof:thm:local-primary:all-estimates:a}\\
        \langle H \widetilde b_i, \widetilde b_i\rangle
        & \geq \big(\frac{1}{2}-\eta_1\big)\|\widetilde b_i\|_2^2,
        && 1\leq i \leq 2^{n+1}-1.
           \label{eq:proof:thm:local-primary:all-estimates:b}
      \end{align}
      Moreover, by~\textrefp[J]{enu:j2} and~\eqref{eq:proof:thm:local-primary:coll-prop:2}, we have
      that
      \begin{equation}\label{eq:proof:thm:local-primary:all-estimates:c}
        \|\widetilde b_i\|_2^2 = \|\widetilde b_I\|_2^2
        \geq (1-\frac{\eta_1}{4}) |I| |C_{[0,1)}|
        \geq (1-\frac{\eta_1}{4}) 2^{-n} |C_{[0,1)}|,
      \end{equation}
      for all $1\leq i \leq 2^{n+1}-1$ and $I\in\mathcal{D}^n$ with $\mathcal{O}(I)=2^i$.
    \end{subequations}
    Comparing~\eqref{eq:proof:thm:local-primary:all-estimates}
    with~\eqref{eq:proof:local-factor:ii}, the only relevant difference is the presence of the
    additional factor $|C_{[0,1)}|$ on the right hand sides
    of~\eqref{eq:proof:thm:local-primary:all-estimates:a}
    and~\eqref{eq:proof:thm:local-primary:all-estimates:c}.  We will now repeat the proof of
    Theorem~\ref{thm:local-factor} with $\widetilde b_i$ in place of $b_i^{(\varepsilon)}$.
    By~\eqref{eq:proof:thm:local-primary:tilde-coll:1} in~\textref[Step~]{step:local-primary:2},
    $(\mathcal{\widetilde B}_I : I\in\mathcal{D}^n)$
    satisfies~\textrefp[J]{enu:j1}--\textrefp[J]{enu:j4} with constant $\kappa_J=(1-\eta_1)^{-2}$.
    Thus, by~\eqref{eq:proof:thm:local-primary:bb:1} and Theorem~\ref{thm:projection}, the operators
    $\widetilde B : SL^\infty_n\to SL^\infty_N$ and $\widetilde Q : SL^\infty_N\to SL^\infty_n$
    given by
    \begin{equation}\label{eq:proof:thm:local-primary:B+Q:operators}
      \widetilde B f
      = \sum_{i=1}^{2^{n+1}-1} \frac{\langle f, h_i\rangle}{\|h_i\|_2^2} \widetilde b_i
      \qquad\text{and}\qquad
      \widetilde Q g
      = \sum_{i=1}^{2^{n+1}-1}
      \frac{\langle g, \widetilde b_i\rangle}{\|\widetilde b_i\|_2^2} h_i
    \end{equation}
    satisfy the estimates
    \begin{equation}\label{eq:proof:thm:local-primary:B+Q:estimates}
      \|\widetilde Bf \|_{SL^\infty} \leq \|f\|_{SL^\infty}
      \qquad\text{and}\qquad
      \|\widetilde Qg \|_{SL^\infty} \leq (1-\eta_1)^{-1}\|g\|_{SL^\infty},
    \end{equation}
    for all $f\in SL^\infty_n$, $g\in SL^\infty_N$.

    Moreover, the diagram
    \begin{equation}\label{eq:proof:thm:projection:diagram}
      \vcxymatrix{SL^\infty_n \ar[rr]^{\Id_{SL^\infty_n}} \ar[rd]_{\widetilde B} & & SL^\infty_n\\
        &  SL^\infty_N  \ar[ru]_{\widetilde Q} &
      }
    \end{equation}
    is commutative.  By~\eqref{eq:proof:thm:local-primary:B+Q:operators}, the bounded projection
    $\widetilde P : SL^\infty_N\to SL^\infty_N$ given by $\widetilde P = \widetilde B \widetilde Q$
    has the form
    \begin{equation}\label{eq:proof:thm:local-primary:P:definition}
      \widetilde P f
      = \sum_{i=1}^{2^{n+1}-1}
      \frac{\langle f, \widetilde b_i\rangle}{\|\widetilde b_i\|_2^2} \widetilde b_i,
      \qquad f\in SL^\infty_N,
    \end{equation}
    and is therefore an orthogonal projection.  By~\eqref{eq:proof:thm:local-primary:B+Q:estimates},
    $\widetilde P$ satisfies the estimate
    \begin{equation}\label{eq:proof:thm:local-primary:P:estimate}
      \|\widetilde P f\|_{SL^\infty}
      \leq (1-\eta_1)^{-1}\|f\|_{SL^\infty},
      \qquad f\in SL^\infty_N.
    \end{equation}

    Now, we define the subspace $\widetilde Y$ of $SL^\infty_N$ by
    \begin{equation*}
      \widetilde Y = \Big\{g = \sum_{i=1}^{2^{n+1}-1} a_i \widetilde b_i :
      a_i\in \mathbb{R}\Big\},
    \end{equation*}
    which is the image of the projection $\widetilde P$.
    By~\eqref{eq:proof:thm:local-primary:B+Q:estimates} and~\eqref{eq:proof:thm:projection:diagram},
    we obtain the following commutative diagram:
    \begin{equation}\label{eq:proof:thm:local-primary:commutative-diagram:preimage}
      \vcxymatrix{SL^\infty_n \ar[r]^{\Id_{SL^\infty_n}} \ar[d]_{\widetilde B} & SL^\infty_n\\
        \widetilde Y \ar[r]_{\Id_{\widetilde Y}} & \widetilde Y \ar[u]_{\widetilde Q|\widetilde Y}},
      \qquad \|\widetilde B\| \|\widetilde Q|\widetilde Y\| \leq (1-\eta_1)^{-1}.
    \end{equation}
    Now, define $\widetilde U : SL^\infty_N\to \widetilde Y$ by
    \begin{equation}\label{eq:proof:thm:local-primary:almost-inverse}
      \widetilde U f = \sum_{i=1}^{2^{n+1}-1}
      \frac{\langle f, \widetilde b_i\rangle}
      {\langle H \widetilde b_i, \widetilde b_i\rangle}
      \widetilde b_i,
      \qquad f\in SL^\infty_N.
    \end{equation}
    and note that by~\eqref{eq:proof:thm:local-primary:all-estimates:b}, the $1$-unconditionality of
    the Haar system in $SL^\infty$ and~\eqref{eq:proof:thm:local-primary:P:estimate}, the operator
    $\widetilde U$ has the upper bound
    \begin{equation}\label{eq:proof:thm:local-primary:almost-inverse:bounded}
      \|\widetilde U : SL^\infty_N\to \widetilde Y\|
      \leq \frac{2}{1-3\eta_1}.
    \end{equation}
    Observe that for all $g = \sum_{i=1}^{2^{n+1}-1} a_i \widetilde b_i \in \widetilde Y$, the
    following identity is true:
    \begin{equation}\label{eq:proof:thm:local-primary:crucial-identity}
      \widetilde UHg - g
      = \sum_{i=1}^{2^{n+1}-1}
      \sum_{j=1}^{i-1} a_j
      \frac{\langle H \widetilde b_j, \widetilde b_i\rangle}
      {\langle H \widetilde b_i, \widetilde b_i\rangle}
      \widetilde b_i
      + \sum_{i=1}^{2^{n+1}-1} \sum_{j=1}^{i-1} a_i
      \frac{\big\langle \widetilde b_i, H^* \widetilde b_j\big\rangle}
      {\langle H\widetilde b_j, \widetilde b_j\rangle}
      \widetilde b_j.
    \end{equation}
    Using $|a_j| \leq \|g\|_{SL^\infty}$ together with~\eqref{eq:proof:local-primary:eta_1},
    \eqref{eq:proof:thm:local-primary:all-estimates}
    and~\eqref{eq:proof:thm:local-primary:crucial-identity} yields
    \begin{equation}\label{eq:proof:thm:local-primary:crucial-inequality}
      \|\widetilde UHg - g\|_{SL^\infty}
      \leq \eta_1 4^{n+3}n  \|g\|_{SL^\infty}.
    \end{equation}
    Finally, let $J : \widetilde Y\to SL^\infty_N$ denote the operator given by $Jy = y$.  By our
    choice of $\eta_1$ in~\eqref{eq:proof:local-primary:eta_1}, the operator
    $\widetilde V : SL^\infty_N\to \widetilde Y$ given by
    $\widetilde V=(\widetilde UHJ)^{-1}\widetilde U$ is well defined.  Thus, the following diagram
    is commutative:
    \begin{equation}\label{eq:proof:thm:local-primary:commutative-diagram:image}
      \vcxymatrix{
        \widetilde Y \ar[rr]^{\Id_{\widetilde Y}} \ar[dd]_J \ar[rd]_{\widetilde UHJ} & & \widetilde Y\\
        & \widetilde Y \ar[ru]^{(\widetilde UHJ)^{-1}} &\\
        SL^\infty_N \ar[rr]_H & & SL^\infty_N \ar[lu]_{\widetilde U} \ar[uu]_{\widetilde V}
      }
      \qquad \|J\|\|\widetilde V\| \leq (1-\eta_1)(1+\eta)/\delta.
    \end{equation}
    The estimate for $\widetilde V$ follows from the choice we made for $\eta_1$
    in~\eqref{eq:proof:local-primary:eta_1}.  Merging the
    diagrams~\eqref{eq:proof:thm:local-primary:commutative-diagram:preimage}
    and~\eqref{eq:proof:thm:local-primary:commutative-diagram:image} concludes the proof.\qedhere
  \end{proofstep}
\end{myproof}

\section{Direct sums of $SL^\infty_n$ spaces are primary}\label{sec:primary}
\noindent
First, we establish that $SL^\infty$ is isomorphic to the $\ell^\infty$ direct sum of its finite
dimensional building blocks $SL^\infty_n$.  Second, we combine the finite dimensional
factorization results Theorem~\ref{thm:local-primary} for $SL^\infty_n$, $n\in\mathbb{N}_0$ to obtain the
factorization result Theorem~\ref{thm:primary} for $\bigl( \sum_{n\in\mathbb{N}_0} SL^\infty_n\bigr)_r$,
$1\leq r\leq \infty$.  Consequently, we obtain that $SL^\infty$ is primary.

\subsection{Isomorphisms and non-isomorphisms of direct sums of
  \bm{$SL^{\infty}_n$}}\label{subsec:isos}\hfill

\noindent
We show that taking direct sums of $SL^\infty_n$ with different parameters produces isomorphically
different spaces, and that $SL^\infty$ is isomorphic to the $\ell^\infty$ direct sum of
$SL^\infty_n$.  We give two proofs for the latter fact: one using the Hahn-Banach theorem, and
another using a compactness argument.  Both rely on Pe{\l}czy{\'n}ski's decomposition
method~\cite{pelczynski:1960}; we refer the reader to~\cite{wojtaszczyk:1991}.

\begin{lem}\label{lem:primary}
  The spaces $\big( \sum_{n\in\mathbb{N}_0} SL^\infty_n \big)_r$, $1\leq r \leq \infty$, are all
  mutually non-isomorphic.  The spaces $SL^\infty$ and
  $\big( \sum_{n\in\mathbb{N}_0} SL^\infty_n \big)_\infty$ are isomorphic.
\end{lem}

\begin{proof}
  By a version of Pitt's theorem for direct sums of finite dimensional Banach spaces (see
  e.g.~\cite[Theorem~5.6]{lechner:2016-factor-mixed} for more details), the spaces
  $\big( \sum_{n\in\mathbb{N}_0} SL^\infty_n \big)_r$, $1\leq r \leq \infty$, are mutually
  non-isomorphic.

  Using Pe{\l}czy{\'n}ski's decomposition method, we will now show that $SL^\infty$ is isomorphic to
  $\big( \sum_{n\in\mathbb{N}_0} SL^\infty_n \big)_\infty$.  First, we will show that $SL^\infty$
  contains a complemented copy of $\big( \sum_{n\in\mathbb{N}_0} SL^\infty_n \big)_\infty$.  To this
  end, consider the collections $\mathcal A_n\subset \mathcal{D}$ given by
  \begin{equation*}
    \mathcal A_n
    = \{ I\in \mathcal{D} : I\subset [1-2^{-n},1-2^{-n-1})\},
    \qquad n\in \mathbb{N}_0,
  \end{equation*}
  and define $E : \big( \sum_{n\in\mathbb{N}_0} SL^\infty_n \big)_\infty\to SL^\infty$ by
  \begin{equation*}
    (f_n)_{n=0}^\infty
    \mapsto \sum_{n=0}^\infty  \chi_{[1-2^{-n},1-2^{-n-1})} f_n\circ \varphi_n,
  \end{equation*}
  where $\varphi_n$ is the affine linear transformation that maps $[1-2^{-n},1-2^{-n-1})$ into
  $[0,1)$, and therefore $\varphi_n(\mathcal{A}_n)=\mathcal{D}$.  Note that
  \begin{align*}
    \|E((f_n)_{n=0}^\infty)\|_{SL^\infty}
    & = \sup_{n\in\mathbb{N}_0} \|\chi_{[1-2^{-n},1-2^{-n-1})} f_n\circ \varphi_n\|_{SL^\infty}
      = \sup_{n\in\mathbb{N}_0} \|f_n\|_{SL^\infty}\\
    & = \big\|(f_n)_{n=0}^\infty\big\|_{(\sum_{n\in\mathbb{N}_0} SL^\infty_n )_\infty},
  \end{align*}
  for all $(f_n)_{n=0}^\infty\in \big( \sum_{n\in\mathbb{N}_0} SL^\infty_n \big)_\infty$, hence, $E$
  is an isometric embedding.  By the $1$-unconditionality of the Haar system in $SL^\infty$, the
  projection $P : SL^\infty\to SL^\infty$ given by
  $f\mapsto \sum_{n=0}^\infty \sum_{I\in\mathcal{A}_n} \langle f, h_I\rangle h_I|I|^{-1}$ has norm
  $1$, and $P(SL^\infty) = E \bigl( ( \sum_{n\in\mathbb{N}_0} SL^\infty_n )_\infty\bigr)$.

  Second, we will now show that $\big( \sum_{n\in\mathbb{N}_0} SL^\infty_n \big)_\infty$ contains
  a complemented copy of $SL^\infty$.  Let
  $G: SL^\infty\to \bigl( \sum_{n\in\mathbb{N}_0} SL^\infty_n \bigr)_\infty$ denote the norm $1$
  operator given by
  \begin{equation*}
    \sum_{I\in\mathcal{D}}a_I h_I\mapsto \Bigl(\sum_{I\in\mathcal{D}^n} a_I h_I\Bigr)_{n=0}^\infty.
  \end{equation*}
  Let $L : \ell^\infty\to \mathbb{R}$ be a norm $1$ functional such that $L(1,1,1,\dots) = 1$.  Now,
  we define the operator $Q : \bigl( \sum_{n\in\mathbb{N}_0} SL^\infty_n \bigr)\to SL^\infty$ by
  \begin{equation*}
    \Bigl(\sum_{I\in\mathcal{D}^n}a_{n,I} h_I\Bigr)_{n=0}^\infty
    \mapsto \sum_{I\in\mathcal{D}} L\bigl( (a_{n,I})_{n=0}^\infty \bigr) h_I,
  \end{equation*}
  and claim that $Q$ has norm $1$.  Let $f_n = \sum_{I\in\mathcal{D}^n} a_{n,I}h_I$,
  $n\in\mathbb{N}_0$ be such that $(f_n)_{n=0}^\infty$ is in the unit ball of
  $\big( \sum_{n\in\mathbb{N}_0} SL^\infty_n \big)_\infty$, i.e.
  $\sup_{n\in\mathbb{N}_0}\|f_n\|_{SL^\infty}\leq 1$.  Assume that
  $\bigr\| Q \bigl((f_n)_{n=0}^\infty\bigr) \bigl\| > 1$.  Thus, there exists a $I_0\in\mathcal{D}$
  such that
  \begin{equation}\label{eq:proof:lem:iso:1}
    \sum_{\substack{I\in\mathcal{D}\\I\supset I_0}}
    \Bigl(L\bigl( (a_{n,I})_{n=0}^\infty \bigr)\Bigr)^2
    > 1.
  \end{equation}
  But then again, we have the estimate
  \begin{align*}
    \sum_{\substack{I\in\mathcal{D}\\I\supset I_0}}
    \Bigl(L\bigl( (a_{n,I})_{n=0}^\infty \bigr)\Bigr)^2
 & = L \Bigl(\sum_{\substack{I\in\mathcal{D}\\I\supset I_0}}L\bigl( (a_{k,I})_{k=0}^\infty \bigr)
    (a_{n,I})_{n=0}^\infty\Bigr)\\
 & \leq \sup_{n\in\mathbb{N}_0}\Bigl|\sum_{\substack{I\in\mathcal{D}\\I\supset I_0}}L\bigl(
    (a_{k,I})_{k=0}^\infty \bigr) a_{n,I}\Bigr|
    \leq \sup_{n,k\in\mathbb{N}_0}\Bigl|\sum_{\substack{I\in\mathcal{D}\\I\supset I_0}}
    a_{k,I}  a_{n,I}\Bigr|\\
 & \leq \sup_{n,k\in\mathbb{N}_0}\Bigl( \sum_{\substack{I\in\mathcal{D}\\I\supset I_0}}
    a_{k,I}^2\Bigr)^{1/2}\Bigl( \sum_{\substack{I\in\mathcal{D}\\I\supset I_0}}
    a_{n,I}^2\Bigr)^{1/2} \leq 1,
  \end{align*}
  which contradicts~\eqref{eq:proof:lem:iso:1}.  Hence, $\|Q\|$ has norm $1$.  By definition of $G$
  and $Q$, the following diagram is commutative:
  \begin{equation}\label{eq:proof:lem:iso:2}
    \vcxymatrix{
      SL^\infty \ar[rr]^{\Id_{SL^\infty}} \ar[dr]_G && SL^\infty\\
      & \big( \sum_{n\in\mathbb{N}_0} SL^\infty_n \big)_\infty \ar[ur]_Q& 
    }
    \qquad \|G\|,\|Q\| = 1.
  \end{equation}
  Consequently, $SL^\infty$ is isomorphic to a complemented subspace of
  $\big( \sum_{n\in\mathbb{N}_0} SL^\infty_n \big)_\infty$, as claimed.

  Finally, since $\big( \sum_{n\in\mathbb{N}_0} SL^\infty_n \big)_\infty$ is isomorphic to
  $\Bigl( \sum_{m\in\mathbb{N}_0}\bigl(\sum_{n\in\mathbb{N}_0} SL^\infty_n \bigr)_\infty
  \Bigr)_\infty$, Pe{\l}czy{\'n}ski's decomposition method (see e.g.\cite{wojtaszczyk:1991}) yields
  that $SL^\infty$ is isomorphic to $\big( \sum_{n\in\mathbb{N}_0} SL^\infty_n \big)_\infty$.
\end{proof}

\begin{rem}
  We want to point out that the construction of the operator $Q$ in the above proof could be
  replaced by a standard argument, which goes back to~\cite[Theorem~III.E.18]{wojtaszczyk:1991},
  \cite{wojtaszczyk:1979}, \cite{johnson:1972} and~\cite{stegall:1973}.  We will now present this
  alternative.  The following proof was taken from~\cite[Theorem~2.2.3]{mueller:2005} and adapted to
  fit our purpose.

  Let $C$ denote the closed unit ball of $SL^\infty$, and let $\mathcal{T}$ denote the smallest
  topology on $C$ such that every functional of the form
  $\langle \cdot, h_I\rangle : C\to \mathbb{R}$, $I\in\mathcal{D}$ is continuous.  We will now prove
  that $(C,\mathcal{T})$ is a compact topological space.  We endow $[-1,+1]^{\mathcal{D}}$ with the
  product topology $\mathcal{P}$, which, by Tychonov's theorem, is a compact topological space, and
  define the map $\Phi : (C,\mathcal{T})\to ([-1,+1]^{\mathcal{D}},\mathcal{P})$ by putting
  \begin{equation*}
    f\mapsto \Big(\frac{\langle f, h_I \rangle}{|I|}\Big)_{I\in\mathcal{D}}.
  \end{equation*}
  By~\eqref{eq:bracket-estimate}, $\Phi$ is well defined, and one can easily check that $\Phi$ is a
  topological embedding.  We will now verify that $\Phi(C)$ is closed.  To this end, let
  $(f_\alpha)$ denote a net in $C$, such that $\Phi(f_\alpha)$ converges to some
  $(a_I)_{I\in\mathcal{D}}\in [-1,+1]^{\mathcal{D}}$.  We define
  $f = \sum_{I\in\mathcal{D}} a_I h_I$, and assume that $\|f\|_{SL^\infty} > 1$.  Then there exists
  an $I_0\in\mathcal{D}$ such that
  \begin{equation*}
    \sum_{\substack{I\in\mathcal{D}\\I\supset I_0}} a_I^2
    > 1
  \end{equation*}
  We now write $(a_{\alpha,I})_{I\in\mathcal{D}} = \Phi(f_\alpha)$ for each $\alpha$, and note that
  by our hypothesis $a_{\alpha,I}\to a_I$, for each $I\in\mathcal{D}$.  Consequently, there exists
  an $\alpha_0$ such that
  \begin{equation*}
    \|f_{\alpha_0}\|_{SL^\infty}^2
    \geq \sum_{\substack{I\in\mathcal{D}\\I\supset I_0}} a_{\alpha_0,I}^2
    > 1,
  \end{equation*}
  which contradicts $f_{\alpha_0}\in C$.  Thus, we know that $\Phi(C)$ is a closed subset of the
  compact topological space $([-1,+1]^{\mathcal{D}},\mathcal{P})$, and is therefore itself compact.
  Since $\Phi$ is a topological embedding, $(C,\mathcal{T})$ is a compact topological space.

  As in the proof of Lemma~\ref{lem:primary}, let
  $G : SL^\infty\to \big( \sum_{n\in\mathbb{N}_0} SL^\infty_n \big)_\infty$ denote the isometric
  embedding given by
  \begin{equation*}
    f\mapsto \Bigl(\sum_{I\in\mathcal{D}^n}\frac{\langle f,h_I\rangle}{|I|}h_I\Bigr)_{n=0}^\infty.
  \end{equation*}
  Let $B$ denote the closed unit ball of $\big( \sum_{n\in\mathbb{N}_0} SL^\infty_n \big)_\infty$,
  and for each $m\in\mathbb{N}_0$ define $R_m : B\to C$ by
  \begin{equation*}
    (f_n)_{n=0}^\infty\mapsto f_m.
  \end{equation*}
  In other words, $R_m\in C^B$, $m\in\mathbb{N}_0$, and since $C^B$ is compact in the product
  topology by Tychonov's theorem, there exists a subnet $(R_\alpha)$ of $(R_m)_{m=1}^\infty$
  converging to, say, $R$.  This means that for each $(f_n)_{n=0}^\infty\in B$, we have that
  $R_\alpha \bigl((f_n)_{n=0}^\infty \bigr)\to R((f_n)_{n=0}^\infty)$ in the topology $\mathcal{T}$
  of $C$.  Now let $f\in C$, then
  \begin{equation*}
    \langle R(Gf), h_I \rangle
    = \lim_\alpha \langle R_\alpha(Gf), h_I \rangle
    = \langle f, h_I\rangle,
    \qquad I\in\mathcal{D}.
  \end{equation*}
  i.e.  $RG|C = \Id_C$.  Certainly, we can extend $R$ to an operator on
  $\big( \sum_{n\in\mathbb{N}_0} SL^\infty_n \big)_\infty$.  We claim that $R$ is bounded by $1$.
  Assume to the contrary, that there exists a $(f_n)_{n=0}^\infty\in B$ with
  $\bigl\|R\bigl((f_n)_{n=0}^\infty\bigr)\bigr\|_{SL^\infty} > 1$.  Then there is an interval
  $I_0\in\mathcal{D}$ such that
  \begin{equation*}
    \sum_{\substack{I\in\mathcal{D}\\I\supset I_0}}
    \bigg( \frac{\bigl\langle R\bigl((f_n)_{n=0}^\infty\bigr), h_I \bigr\rangle}{|I|} \bigg)^2
    > 1.
  \end{equation*}
  Since $R_\alpha \bigl((f_n)_{n=0}^\infty \bigr)\to R((f_n)_{n=0}^\infty)$ in $(C,\mathcal{T})$,
  there exists an $\alpha_0$ such that
  \begin{equation*}
    1 < \sum_{\substack{I\in\mathcal{D}\\I\supset I_0}} \bigg(
    \frac{\bigl\langle R_{\alpha_0}\bigl((f_n)_{n=0}^\infty\bigr), h_I \bigr\rangle}{|I|}
    \bigg)^2
    = \bigl\| R_{\alpha_0}\bigl((f_n)_{n=0}^\infty\bigr) \bigr\|_{SL^\infty}^2
    = \| f_{\alpha_0} \|_{SL^\infty}^2,
  \end{equation*}
  which contradicts $(f_n)_{n=0}^\infty\in B$.  Thus, $R$ has norm $1$.

  The above $R$ would be a suitable replacement for $Q$ in the proof of Lemma~\ref{lem:primary}.
\end{rem}

\subsection{Proof of Theorem~\ref{thm:primary}}\label{subsec:primary}\hfill

\noindent
When using Bourgain's localization method, see e.g.~\cite{bourgain:1983, blower:1990, mueller:2005,
  wark:2007, mueller:2012, lechner:mueller:2014}, we eventually have to pass from the local result
to the global result; in our case from Theorem~\ref{thm:local-primary} to Theorem~\ref{thm:primary} below, which
includes diagonalizing operators on the sum of the finite dimensional spaces.  For us, these sums
are $\bigl(\sum_{n\in\mathbb{N}_0} SL^\infty_n \bigr)_r$, $1\leq r \leq \infty$.  Due to a gliding
hump argument, this diagonalization process does not require any additional hypothesis if
$r < \infty$; if $r=\infty$ however, we require that the sequence of spaces
$(SL^\infty_n)_{n\in\mathbb{N}_0}$ has the property that projections almost annihilate finite
dimensional subspaces (which it has), see Definition~\ref{dfn:property-pafds},
Corollary~\ref{cor:projections-that-annihilate} and Remark~\ref{rem:property-pafds}.  For more details, we refer
the reader to~\cite[Section 5]{lechner:2016-factor-mixed}.

\begin{proof}[Proof of Theorem~\ref{thm:primary}]
  By Theorem~\ref{thm:local-primary}, the first part of the hypothesis
  of~\cite[Proposition~5.4]{lechner:2016-factor-mixed} is satisfied. Moreover by
  Remark~\ref{rem:property-pafds} for each $\eta > 0$, the sequence of finite dimensional Banach spaces
  $(SL^\infty_n)_{n\in\mathbb{N}_0}$ has the property that projections almost annihilate finite
  dimensional subspaces with constant $1+\eta$ (which is only needed in the case $r=\infty$). Thus,
  applying~\cite[Proposition~5.4]{lechner:2016-factor-mixed} yields the
  diagram~\eqref{eq:thm:primary}.

  It is easily verified that for fixed $1\leq r\leq \infty$, the Banach space $X^{(r)}$ is
  isomorphic to $\bigl( \sum_{m\in\mathbb{N}} X^{(r)} \bigr)_r$, hence, by Pe{\l}czy{\'n}ski's
  decomposition method (see e.g.~\cite{wojtaszczyk:1991}) and diagram~\eqref{eq:thm:primary}, we
  obtain that $X^{(r)}$ is primary.

  Finally, by Lemma~\ref{lem:primary}, $SL^\infty$ is isomorphic to the primary Banach space
  $X^{(\infty)}$, and is thereby itself primary.
\end{proof}

\subsection*{Acknowledgments}\hfill\\
\noindent
It is my pleasure to thank P.F.X.~Müller for many helpful discussions.  Supported by the Austrian
Science Foundation (FWF) Pr.Nr. P28352.

\bibliographystyle{abbrv}
\bibliography{bibliography}

\end{document}